\documentclass[12pt]{amsart}

\usepackage{amsmath, amssymb, amsthm, mathrsfs, graphicx, hyperref,float}
\usepackage[shortlabels]{enumitem}
\usepackage[all,cmtip]{xy}
\usepackage[numbers]{natbib}
\usepackage[margin=1in]{geometry}
\usepackage{tikz-cd}
\usepackage{todonotes}

\usepackage{yfonts}

\graphicspath{{graphics/}}

\newtheorem{theorem}{Theorem}
\newtheorem{proposition}[theorem]{Proposition}
\newtheorem{lemma}[theorem]{Lemma}
\newtheorem{corollary}[theorem]{Corollary}

\theoremstyle{definition}
\newtheorem{definition}[theorem]{Definition}
\newtheorem{example}[theorem]{Example}
\newtheorem{remark}[theorem]{Remark}

\numberwithin{theorem}{section}

\makeatletter
\providecommand{\leftsquigarrow}{%
  \mathrel{\mathpalette\reflect@squig\relax}%
}
\newcommand{\reflect@squig}[2]{%
  \reflectbox{$\m@th#1\rightsquigarrow$}%
}
\makeatother

\renewcommand{\phi}{\varphi} 

\newcommand{\lev}{\mathrm{lev}}
\newcommand{\ol}[1]{\overline{#1}}

\definecolor{sebgreen1}{rgb}{0.019,0.317,0.149}
\definecolor{sebgreen2}{rgb}{0.784,0.952,0.780}
\definecolor{refred1}{rgb}{1,0,0}
\definecolor{refred2}{rgb}{1,.8,.8}

\newcommand{\A}{\mathbb{A}}

\newcommand{\G}{\mathbb{G}}
\newcommand{\N}{\mathbb{N}}

\newcommand{\Z}{\mathbb{Z}}

\newcommand{\Part}{\mathrm{part}}
\newcommand{\PPart}{\mathrm{Part}}

\newcommand{\PL}{\mathrm{PL}}

\newcommand{\Spec}{\mathrm{Spec}\,}

\newcommand{\trop}{\mathrm{trop}}

\title[Gorenstein modular compactifications of $\ol{\mathcal{M}}_{1,n}$]{A classification of modular compactifications of the space of pointed elliptic curves by Gorenstein curves}
\author{Sebastian Bozlee, Bob Kuo, and Adrian Neff}

\begin{document}

\maketitle

\begin{abstract}
We classify the Deligne-Mumford stacks $\mathcal{M}$ compactifying the moduli space $\mathcal{M}_{1,n}$ of smooth $n$-pointed curves of genus one under the condition that the points of $\mathcal{M}$ represent Gorenstein curves with
distinct smooth markings. This classification uncovers new moduli spaces $\ol{\mathcal{M}}_{1,n}(Q)$, which we may think of coming from an enrichment of the notion of level used to define Smyth's $m$-stable spaces. Finally, we construct a cube complex of Artin stacks interpolating between the $\ol{\mathcal{M}}_{1,n}(Q)$'s, a multidimensional analogue of the wall-and-chamber structure seen in the log minimal model program for $\ol{\mathcal{M}}_g$.
\end{abstract}

\maketitle

\section{Introduction}

The moduli stack $\mathcal{M}_{g,n}$ of smooth genus $g$ algebraic curves with $n$ marked points is not proper, so one searches for compactifications, that is, proper Deligne-Mumford stacks $\mathcal{M}$ so that $\mathcal{M}_{g,n}$ embeds as a dense open substack of $\mathcal{M}$.
In this paper we construct a new family of modular compactifications of $\mathcal{M}_{1,n}$. We then show that these moduli spaces exhaust the semistable modular compactifications of $\mathcal{M}_{1,n}$ with Gorenstein singularities and distinct markings.

Let us now set up some notation in order to give the definition of these new moduli spaces.

\begin{definition}
Given a positive integer $n$, let $\PPart(n)$ be the set of partitions of $\{1, \ldots, n \}$. Give $\PPart(n)$ a partial order by $P_1 \preceq P_2$ if the partition $P_1$ is refined by the partition $P_2$.

Denote by $\mathfrak{Q}_n$ the collection of subsets $Q \subseteq \PPart(n)$ such that
\begin{enumerate}
    \item $Q$ is downward closed;
    \item $Q$ does not contain the discrete partition $\{ \{1\}, \ldots, \{n\} \}$. 
\end{enumerate}
\end{definition}

\begin{definition}
Let $p$ be a closed point of an algebraic curve $C$ over an algebraically closed field $k$, and let $\nu : \tilde{C} \to C$ be the normalization.
The \emph{number of branches at $p$} is
\[
  m(p) = |\nu^{-1}(p)|.
\]
The \emph{delta invariant of $p$} is
\[
  \delta(p) = \dim_{k}(\nu_*\mathscr{O}_{\tilde{C}}/\mathscr{O}_C).
\]
The \emph{genus of $p$} is
\[
  g(p) = \delta(p) - m(p) - 1.
\]
\end{definition}

The genus of a singularity captures its contribution to the genus of $C$. In particular, if $C$ is connected and proper, $p_1, \ldots, p_e$ are the singularities of $C$, and $\tilde{C}_1,\ldots, \tilde{C}_v$ are the irreducible components of $\tilde{C}$, then the arithmetic genus of $C$ is
\[
  g(C) = \sum_{i = 1}^e g(p_i) + \sum_{j = 1}^v g(\tilde{C}_j) + b_1(\Delta_C),
\]
where $b_1(\Delta_C)$ is the first Betti number of the simplicial complex $\Delta_C$ with vertices $\tilde{C}_1,\ldots, \tilde{C}_v$ and, for each $p_i$, an $(m(p_i) - 1)$-simplex whose vertices are glued to the components meeting $\nu^{-1}(p_i)$.

\begin{definition}
A closed point $p$ of an algebraic curve $C$ over an algebraically closed field is an \emph{elliptic Gorenstein singularity} if $\mathscr{O}_{C,p}$ is Gorenstein and $g(p) = 1.$
\end{definition}

It is shown in \cite{smyth_mstable} that the elliptic Gorenstein singularities are classified by their number of branches, $m$. If
$m = 1$, $p$ is a cusp; for $m = 2$, $p$ is a tacnode; for $m \geq 3$,
$p$ is the union of the coordinate axes of $\mathbb{A}^{m-1}$ with one more line transverse to each of the coordinate hyperplanes of $\mathbb{A}^{m-1}$. Given such a singularity, we will call the irreducible components to which $p$ belongs the \emph{branches} of $p$.

\begin{definition}
A \emph{subcurve} $Z$ of a proper algebraic curve $C$ over an algebraically closed field is a connected reduced closed subscheme of $C$.

Let $(C, p_1, \ldots, p_n)$ be a curve of arithmetic genus one together with $n$ marked closed points over an algebraically closed field. Let $Z$ be a subcurve of $C$ of genus one and let $\Sigma$ be the divisor of markings. We define the \emph{level of $Z$}, $\lev(Z)$, to be the partition of $\{1, \ldots, n \}$
where $a, b \in \{1, \ldots, n \}$ lie in the same subset if and only if the markings $p_a, p_b$ lie in the same connected component of $(C - Z) \cup \Sigma$.

If $q \in C$ is an elliptic Gorenstein singularity, we say the \emph{level of $q$}, $\lev(q)$ is the partition of $\{ 1, \ldots, n \}$ where $a, b \in \{1, \ldots, n \}$ lie in the same subset if and only if the markings $p_a, p_b$ lie in the same connected component of the normalization of $C$ at $q$. (i.e. if the rational tails
containing $p_a, p_b$ are connected via a nodal path to the same branch of the singularity).
\end{definition}

\begin{remark}
If $Z_1$ and $Z_2$ are two genus one subcurves of $C$ and $Z_1 \subseteq Z_2$, then $\lev(Z_1) \preceq \lev(Z_2)$.
\end{remark}

\begin{remark}
The level of $C$ considered in \cite{smyth_mstable} is the cardinality $|\lev(Z)|$, where $Z$ is the minimal subcurve of $C$ of genus one.

The level of an elliptic Gorenstein singularity $q$ defined here was called the ``combinatorial type of $C$" in \cite[Definition 2.15]{smyth_mstable2}.
\end{remark}

\begin{definition}
Let $Q \in \mathfrak{Q}_n$. A $Q$-stable curve over a scheme $S$ consists of
\begin{enumerate}
  \item $\pi : C \to S$ a flat and proper morphism of schemes;
  \item $\sigma_1, \ldots, \sigma_n: S \to C$ sections of $\pi$ with disjoint images
\end{enumerate}
such that for each geometric fiber $(C_s, p_1, \ldots, p_n)$, 
\begin{enumerate}
  \item $C_s$ is a connected reduced Gorenstein scheme of dimension 1 with arithmetic genus one;
  \item (level condition on subcurves) If $Z \subseteq C_s$ is a subcurve of genus one, then $\lev(Z) \not\in Q$;
  \item (level condition on singularities) If $q \in Z$ is a genus one singularity, then $\lev(q) \in Q$;
  \item $H^0(C, \Omega^\vee_C(-\Sigma)) = 0$.
\end{enumerate}

We define the \emph{moduli space $\ol{\mathcal{M}}_{1,n}(Q)$ of $Q$-stable $n$-marked curves of genus one} to be the stack over $\Z[1/6]$ whose $S$-points are the $Q$-stable curves over $S$.
\end{definition}

Our first main result is that this defines a modular compactification of $\mathcal{M}_{1,n}$.

\begin{theorem}[= Theorem \ref{thm:construction}]
For each $Q \in \mathfrak{Q}_n$, $\ol{\mathcal{M}}_{1,n}(Q)$ is a proper irreducible Deligne-Mumford stack over $\Z[1/6]$ containing $\mathcal{M}_{1,n}$.
\end{theorem}

When $Q = \{ S \in \PPart(n) : |S| \leq m \}$ for some $m$, we recover Smyth's $m$-stable compactification $\ol{\mathcal{M}}_{1,n}(m)$\cite{smyth_mstable}. We may regard the spaces $\ol{\mathcal{M}}_{1,n}(Q)$ as ``combinatorial remixes''
of the $m$-stable spaces, since each of the curves of $\ol{\mathcal{M}}_{1,n}(Q)$ for some $Q$ belong to some $\ol{\mathcal{M}}_{1,n}(m)$ for various $m$. Despite this, the $Q$-stable spaces are surprisingly plentiful: for $n = 5$, there are only $5$
$m$-stable spaces, but $79,814,831$ $Q$-stable spaces.

All of the $Q$-stable spaces arise from compatible choices of how to contract the universal curve of the moduli space of radially aligned log curves (defined in \cite{rsw} and \cite{keli_thesis}), analogously to the ``extremal assignments'' of \cite{smyth_zstable}. It was systematic enumeration of
such contractions using the log-geometric techniques of \cite{bozlee_thesis} that led to the discovery of the $Q$-stable spaces.
This leads to a resolution of the rational map between the Deligne-Mumford-Knudsen space and each $Q$-stable
space.

\begin{theorem}[= Theorem \ref{thm:contraction_to_Q}] \label{thm:contraction_to_Q_intro}
For each $Q \in \mathfrak{Q}_n$, there is a diagram of stacks
\[
  \xymatrix{
     & \ol{\mathcal{M}}_{1,n}^{rad} \ar[dl] \ar[dr] & \\
    \ol{\mathcal{M}}_{1,n}  & & \ol{\mathcal{M}}_{1,n}(Q)
  }
\]
so that both arrows are proper and restrict to an isomorphism on $\mathcal{M}_{1,n}$.
\end{theorem}

We will also find the construction of contractions of families of curves to be helpful sporadically throughout the paper.

Our next main theorem is that the $Q$-stable spaces account for all sufficiently nice modular compactifications in genus one, taking us one step further in the classification of modular compactifications of the moduli space of pointed algebraic curves. To that end, we introduce some definitions.

\begin{definition}
Let $\mathcal{U}_{1,n}$ be the stack of Gorenstein, connected, reduced curves of genus one with $n$ distinct smooth marked points and no infinitesimal automorphisms. For the purposes of this paper, a \emph{modular compactification} is an open Deligne-Mumford substack $\mathcal{M}$ of $\mathcal{U}_{1,n}$, proper over $\Spec \mathbb{Z}[\frac{1}{6}]$.
\end{definition}

In the language of \cite{smyth_zstable}, a modular compactification in our sense is a semistable modular compactification whose curves are Gorenstein with distinct smooth markings, except that the base is chosen as $\Spec \mathbb{Z}[\frac{1}{6}]$ instead of $\Spec \mathbb{Z}$.

\begin{theorem} \label{thm:classification}
If $\mathcal{M}$ is a modular compactification of $\mathcal{M}_{1,n}$, then $\mathcal{M} \cong \ol{\mathcal{M}}_{1,n}(Q)$ for some $Q$.
\end{theorem}

We prove this classification theorem over the course of Section \ref{sec:classification}.

Finally, in Section \ref{sec:interpolation}, we construct a cube complex of mildly non-separated Artin stacks interpolating between the $\ol{\mathcal{M}}_{1,n}(Q)$'s. This complex yields a multidimensional analogue of the wall-and-chamber structure seen in the log minimal model program for $\ol{\mathcal{M}}_g$.

This paper gives the first general classification of Gorenstein modular compactifications of $\mathcal{M}_{g,n}$ in genus greater than 0. In future work we hope to use similar ideas to construct and classify modular compactifications of $\mathcal{M}_{g,n}$. For instance, Battistella \cite{luca} has constructed a sequence of modular compactifications of $\mathcal{M}_{2,n}$ parametrized by a level analogous to that of \cite{smyth_mstable}, and our more flexible notion of level should also yield combinatorial variations of Battistella's moduli spaces.

It would also be natural to search for similar results on modular compactifications in which the marked points are permitted to come together, as in
the spaces of weighted stable curves of \cite{hassett_weighted_curves}. The thesis of Andy Fry\cite{fry_thesis} suggests that it is necessary to consider more general collisions of markings than those permitted by weights. This will be pursued in future work with Vance Blankers.

\subsection{Acknowledgements}

We would like to thank David Smyth for suggesting that we look for ``universal mesas" which turned into the universal radii seen here, Jonathan Wise, who jointly supervised the summer research project that led to these results, Sam Scheeres and Toby Aldape who also participated in the project, and Connor Meredith and Mathieu Foucher, whose software for enumerating the strata of $\ol{M}_{g,n}^{trop}$ was very useful for exploring the space of universal mesas, Dhruv Ranganathan for his comments on an early version of this paper, and the anonymous referee whose comments improved its presentation. The first author
would also like to thank Leo Herr for some helpful conversations.

\section{Examples}

In this section we give some examples to illustrate the nature and variety of $Q$-stable spaces.
We start by describing how to count $Q$-stability conditions.

\begin{definition}
An \emph{antichain} in a partially ordered set $P$ is a subset $A \subseteq P$ so that no distinct elements of $A$ are comparable.
\end{definition}

\begin{proposition}
Let $P$ be a partially ordered set. There is a bijection
\[
  \left\{ Q \subsetneq P : Q \text{ downward closed} \right\} \longleftrightarrow \left\{ A \subseteq P : \text{$A$ is a non-empty antichain} \right\}
\]
given left-to-right by taking $Q$ to the set of minimal elements of $P - Q$, and right-to-left by taking $A$ to the complement of the upward closure of $A$.
\end{proposition}
\begin{proof}
Omitted.
\end{proof}

The number of non-empty antichains of the lattice of partitions of $n$ elements are counted in OEIS sequence A302251 \cite{oeis_antichains}.
We learn that there are
\begin{itemize}
  \item 9 $Q$-stable compactifications of $\mathcal{M}_{1,3}$,
  \item 346 $Q$-stable compactifications of $\mathcal{M}_{1,4}$,
  \item 79,814,831 $Q$-stable compactifications of $\mathcal{M}_{1,5}$.
\end{itemize}
By contrast, for a given $n$, there are only $n$ compactifications of $\mathcal{M}_{1,n}$ by $m$-stable spaces.

\medskip

Since the properties of being downward closed and of being a proper subset are preserved by finite union and intersection, the set $\mathfrak{Q}_n$ forms a lattice under union and intersection.

\begin{example}
Consider the case $n =3$. The Hasse diagram of $\PPart(3)$ and the corresponding lattice of $Q$-stability conditions for $n = 3$ are displayed in Figure \ref{fig:q_stable_lattice_3}. Visually, $\mathfrak{Q}_3$ consists of a cube
and a whisker: we will show later that the lattice is always a ``union of cubes" and consider a way to fill in the interior of the cube.

We see that there are 9 $Q$-stable spaces for $n = 3$, in agreement with the count just above. Three of those are $m$-stable spaces: $\ol{\mathcal{M}}_{1,3}$ corresponds to the subset at
the bottom of the diagram, $\ol{\mathcal{M}}_{1,3}(1)$ to the subset just above, and $\ol{\mathcal{M}}_{1,3}(2)$ to the subset at the top of whole diagram.

In Figure \ref{fig:stable_curves_3} we give some examples of $3$-pointed curves and the $Q$'s for which they are considered stable.

\begin{figure}[h]
\begin{center}
\includegraphics[height=3.5in]{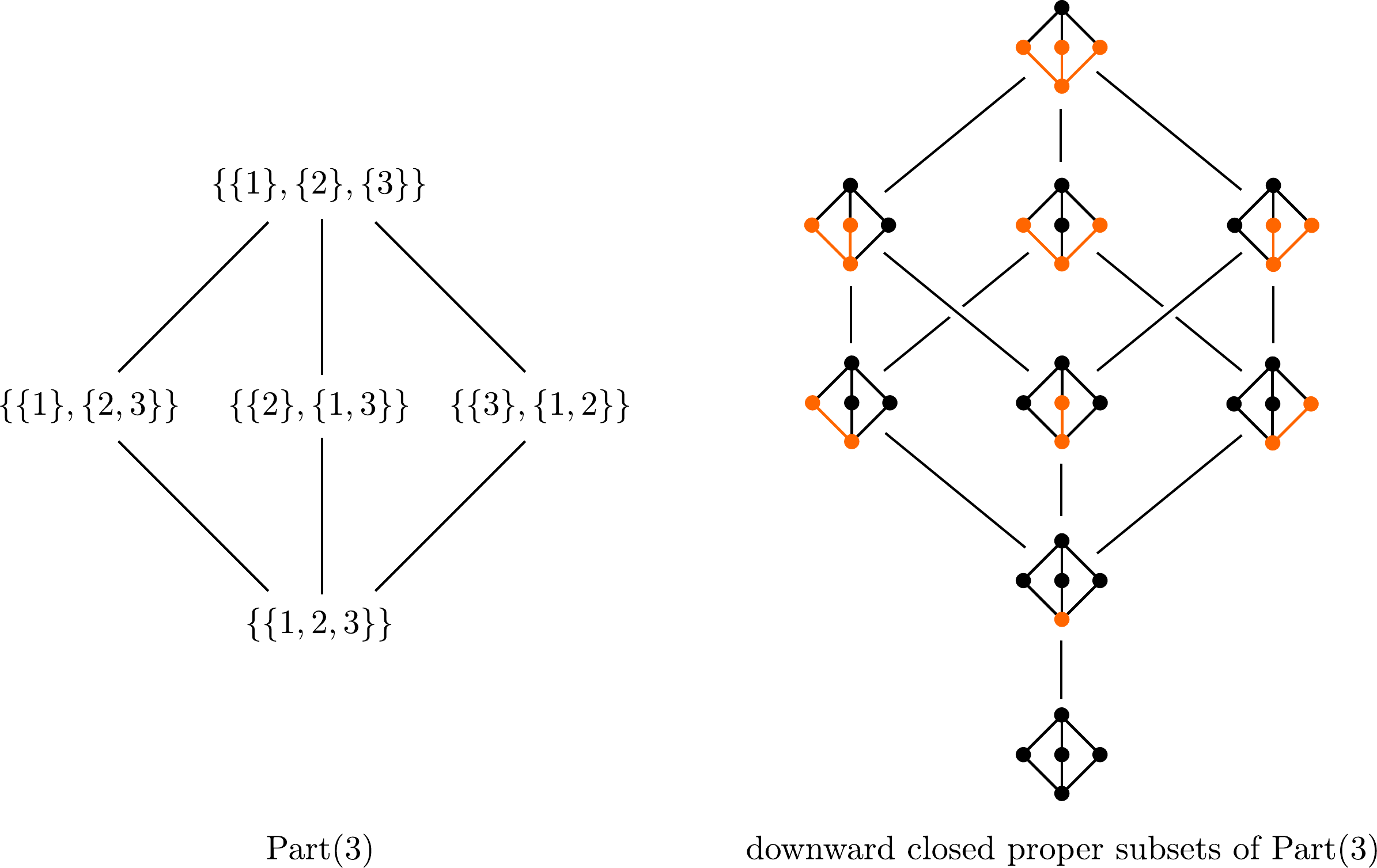}
\end{center}
\caption{The partially ordered set of partitions of $\{1,2,3\}$ and the lattice of $Q$-stability conditions for $n = 3$. An orange dot indicates that the corresponding partition on the left is included in $Q$.} \label{fig:q_stable_lattice_3}
\end{figure}

\begin{figure}[h]
\begin{center}
  \includegraphics[width=\textwidth]{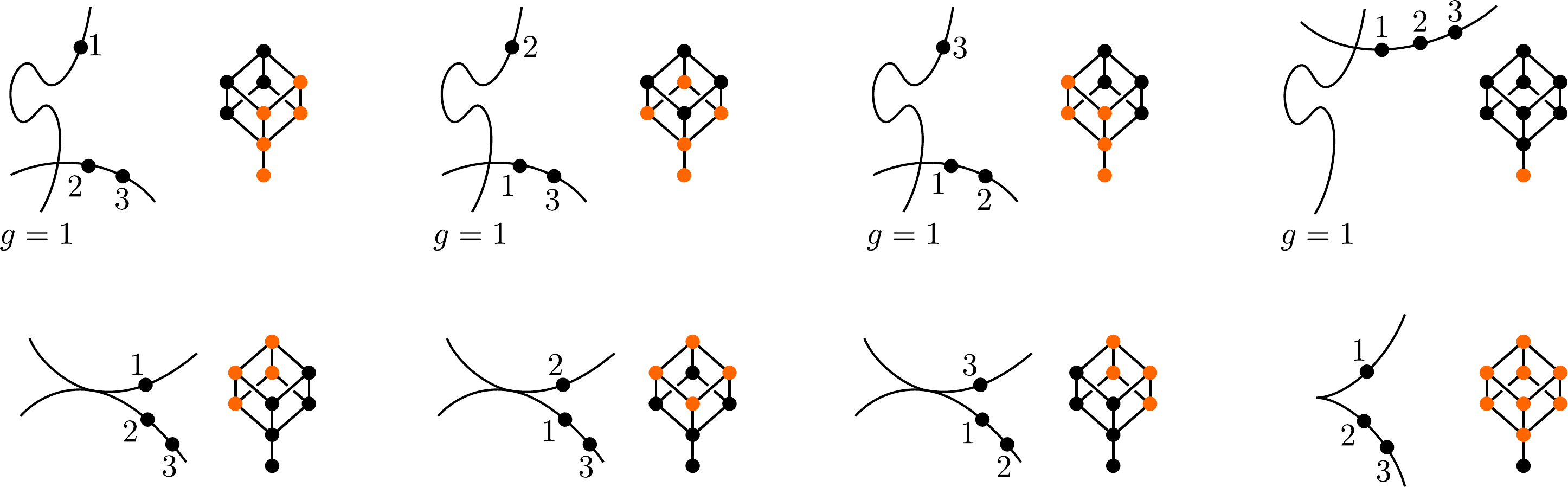}
\end{center}
\caption{Some Gorenstein $3$-pointed curves of genus one. Next to each curve we indicate for which choices of $Q$ the curve is $Q$-stable: an orange dot
means that the curve is stable with respect to the corresponding point in the diagram of downward closed subsets in Figure \ref{fig:q_stable_lattice_3}.} \label{fig:stable_curves_3}
\end{figure}

\end{example}

\begin{example}
For $n = 3$ none of the new stability conditions---that is, the $Q$'s so that $\ol{\mathcal{M}}_{1,n}(Q)$ is not an $m$-stable space---are symmetric with respect
to the markings. This is a coincidence for low $n$.

Say that a proper downward closed subset $Q$ of $\PPart(n)$ is \emph{symmetric} if $Q$ is fixed by the natural $S_n$ action. The orbits
of partitions of $\{ 1, \ldots, n \}$ are in bijection with the integer partitions of $n$, so we may equivalently think of a symmetric $Q$ as a proper downward
closed subset of the partially ordered set of integer partitions of $n$ ordered by refinement. Since the property of being symmetric is preserved under
intersection and union, the set of symmetric $Q$-stable conditions for $n$ also form a lattice under union and intersection.

Consider the set of symmetric $Q$-stability conditions when $n = 5$. The Hasse diagram of the integer partitions of 5 and the lattice of symmetric $Q$-stable conditions, colored by the corresponding subset of the integer partitions of 5, is
shown in Figure \ref{fig:symmetric_Qs_5}. The 5 $m$-stable spaces are given by the subsets down the middle of the diagram on the right; the 4 remaining subsets
on the sides of that diagram yield new moduli spaces.

\begin{figure}[h]
\begin{center}
\includegraphics[height=3in]{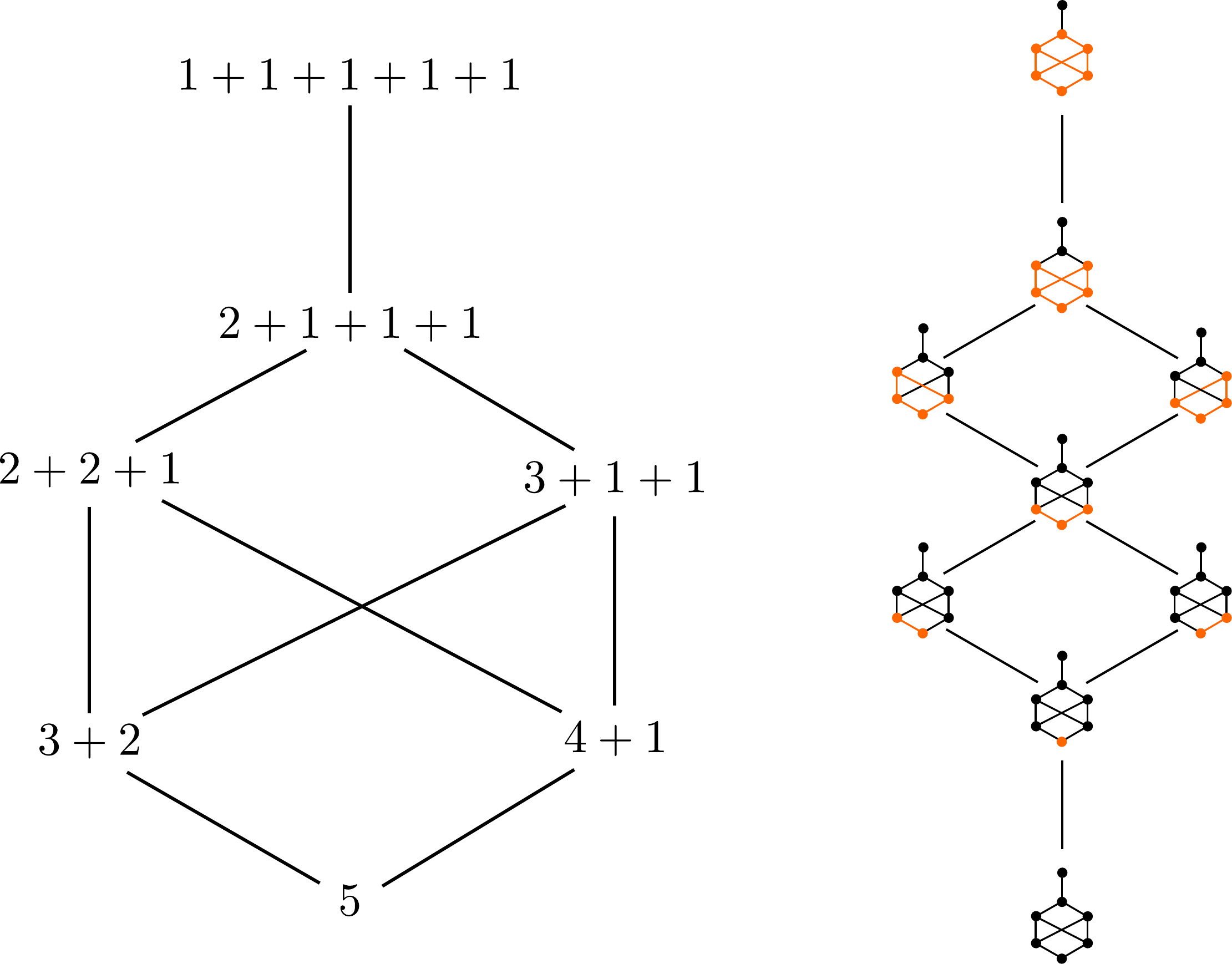}
\end{center}
\caption{On the left, the integer partitions of $5$. On the right, the lattice of symmetric $Q$-stability conditions when $n = 5$, thought of as subsets of the set of integer partitions. An orange dot indicates that the corresponding integer partition in the diagram on the left is included in $Q$.} \label{fig:symmetric_Qs_5}
\end{figure}

\end{example}

\section{Preliminaries}

\subsection{Tropical curves}

Our main tool is the log geometric approach to tropical geometry. We will use the framework of \cite{ccuw}. All of our monoids will be commutative and we take $\mathbb{N}$ to include zero. We will prefer additive notation for the operation of $P$.

Recall that a monoid $P$ is
\begin{enumerate}
  \item \emph{sharp} if its only invertible element is the identity,
  \item \emph{integral} if $a + b = a + c$ implies $b = c$ for all $a,b,c \in P$,
  \item \emph{finitely generated} if there is a surjection $\N^r \to P$ for some integer $r$,
  \item \emph{saturated} if $P$ is integral and for any $a \in P^{gp}$ and $n \in \Z_{> 0}$, $n \cdot a \in P$ implies $a \in P$.
  \item \emph{fs} if $P$ is finitely generated, integral, and saturated.
\end{enumerate}

We begin by recalling the definition of tropical curve, which is essentially a graph whose edges are labeled with ``lengths'' from an fs sharp monoid.

\begin{definition}
An $n$-marked \emph{tropical curve} $\Gamma$ \emph{with edge lengths in a fs sharp monoid} $P$ consists of:
\begin{enumerate}
  \item A finite set $X(\Gamma) = V(\Gamma) \sqcup F(\Gamma)$. The elements of $V(\Gamma)$ are called the \emph{vertices} of $\Gamma$ and the elements of $F(\Gamma)$ are called the \emph{flags} of $\Gamma$.
  \item A \emph{root map} $r_\Gamma : X(\Gamma) \to X(\Gamma)$ which is idempotent with image $V(\Gamma)$.
  \item An involution $\iota_\Gamma : X(\Gamma) \to X(\Gamma)$ that fixes $V(\Gamma)$. The subsets $\{ f, \iota_\Gamma(f) \}$ of $F(\Gamma)$
  of size two are called \emph{edges}, and the set of all edges is denoted $E(\Gamma)$. The subsets $\{ f, \iota_\Gamma(f) \}$ of $F(\Gamma)$ of size one are called \emph{legs},
  and the set of all legs is denoted $L(\Gamma)$.
  \item A bijection $l : \{ 1, \ldots, n \} \to L(\Gamma)$.
  \item A function $g : V(\Gamma) \to \N$. Given a vertex $v$, $g(v)$ is called the \emph{genus of $v$}.
  \item A function $\delta : E(\Gamma) \to P$. Given an edge $e$, $\delta(e)$ is called the \emph{length of $e$}.
\end{enumerate}
\end{definition}

We imagine that each flag $f$ is half of an edge starting at the vertex $r_\Gamma(f)$. Given an edge $e = \{ f, \iota_\Gamma(f) \}$, we say the vertices $r_\Gamma(f)$ and $r_\Gamma(\iota_\Gamma(f))$ are \emph{incident} to $e$.

\begin{definition}
The \emph{genus} of a tropical curve $\Gamma$ is
\[
  g(\Gamma) = b_1(\Gamma) + \sum_{v \in V(\Gamma)} g(v),
\]
where $b_1(\Gamma)$ is the first Betti number of $\Gamma$, that is, $|E(\Gamma)| - |V(\Gamma)| + n$ where $n$ is the number of connected components of $\Gamma$.
\end{definition}

\begin{definition}
A tropical curve is \emph{stable} if it is connected, not an isolated vertex of genus one, and the valence of each of its vertices of genus 0 is at least 3.
\end{definition}

\begin{definition}
A \emph{piecewise linear function} $f$ on a tropical curve $\Gamma$ with edge lengths in $P$ consists of
\begin{enumerate}
  \item a value $f(v) \in P$ for each vertex $v \in V(\Gamma)$;
  \item a slope $m(l) \in \mathbb{N}$ for each leg $l \in L(\Gamma)$
\end{enumerate}
such that whenever $e$ is an edge with ends $v$ and $w$, $f(v) - f(w)$ is an integer multiple of $\delta(e)$.

The set of all piecewise linear functions on $\Gamma$ is denoted $\PL(\Gamma)$.
\end{definition}

Given a tropical curve $\Gamma$ with edge lengths in $P$ and a morphism of fs sharp monoids $\pi^\sharp : P \to P'$, we may apply $\pi^\sharp$ to the edge lengths of $\Gamma$ and contract edges of length zero
to arrive at a new tropical curve. Composing with an isomorphism gives us the notion of a weighted edge contraction, which we define below.

\begin{definition}
Let $\Gamma, \Gamma'$ be tropical curves with edge lengths in $P$ and $P'$, respectively. A \emph{weighted edge contraction} $\pi: \Gamma' \to \Gamma$ (note the variance!) consists of
\begin{enumerate}
  \item a function $\pi : X(\Gamma) \to X(\Gamma')$
  \item a morphism of monoids $\pi^\sharp : P \to P'$
\end{enumerate}
such that
\begin{enumerate}
  \item $\pi$ preserves ends of flags: $\pi \circ r_\Gamma = r_{\Gamma'} \circ \pi$;
  \item $\pi$ preserves edges: $\pi \circ \iota_\Gamma = \iota_{\Gamma'} \circ \pi$;
  \item $\pi$ sends legs of $\Gamma$ bijectively to legs of $\Gamma'$ and preserves their markings;
  \item for each flag $f \in F(\Gamma')$, the preimage $\pi^{-1}(f)$ has exactly one element (automatically a flag);
  \item for each vertex $v \in V(\Gamma')$, the preimage $\pi^{-1}(v)$ is a connected weighted graph of genus $g(v)$;
  \item the flags of an edge $e \in E(\Gamma)$ are sent by $\pi$ to a vertex of $\Gamma'$ if and only if $\pi^\sharp(\delta(e)) = 0$;
  \item for each edge $e \in E(\Gamma)$ with $\pi^{\sharp}(\delta(e)) \neq 0$, the image of $e$ is an edge $e'$ of length $\delta(e') = \pi^\sharp(\delta(e))$.
\end{enumerate}

We will call a weighted edge contraction a \emph{face contraction} if there is a subset $S \subseteq P$ so that the map $\pi^\sharp$ is of the form
\[
  P \longrightarrow S^{-1}P \longrightarrow S^{-1}P / (S^{-1}P)^* \overset{\sim}{\longrightarrow} P'
\]
where the first arrow is localization, the second is the quotient by the submonoid of invertible elements, and the third is an isomorphism.
(These are the edge contractions associated to face inclusions in the category of rational polyhedral cones\cite[Definition 2.25]{ccuw}.) In the case that $P$ is a finite free monoid $\N^r$, the face contractions are precisely the projections of $\N^r$ onto subsets of its coordinates.

Given a weighted edge contraction $\pi : \Gamma' \to \Gamma$ there is an induced map
\[
  \pi^*: \PL(\Gamma) \to \PL(\Gamma')
\]
given by taking $f$ with values $f(v)$ and slopes $m(l)$ to the piecewise linear function $\pi^*f$ with values $(\pi^*f)(v) = \pi^\sharp(f(v))$ for $v \in V(\Gamma)$ and the same slopes.

\end{definition}

We take weighted edge contractions to be the morphisms in the category of tropical curves. In particular, an isomorphism of tropical curves is an 
invertible weighted edge contraction.

\subsection{Log curves and their tropicalizations}

The natural notion of family of curves in logarithmic geometry admits both an underlying family of pointed nodal curves and a tropicalization, connecting
the tropical and algebro-geometric worlds.
F. Kato introduced the notion of a family of log curves in \cite{fkato_deformations}.

\begin{definition} (cf. \cite[Definition 1.2]{fkato_deformations}) Let $S$ be an fs log scheme. A \emph{log curve} over $S$ is a log smooth and integral morphism $\pi : C \to S$ of fs log schemes such that every geometric fiber of $\pi$ is
a reduced and connected curve.
\end{definition}

\begin{figure} 
\centering
  \includegraphics{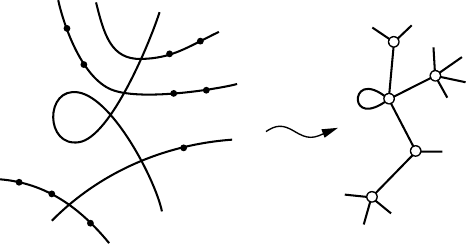}
  \caption{A typical log curve and its tropicalization.}
  \label{fig:dual_graph_example}
\end{figure}

F. Kato has shown that the underlying morphism of schemes of a log curve is a family of nodal curves, and the data in the log structure records some marked points \cite[Theorem 1.3]{fkato_deformations}. We borrow this statement of F. Kato's local structure theorem from \cite{rsw}.

\begin{theorem} \label{thm:log_curve_local_structure}
Let $\pi : C \to S$ be a family of proper log curves. If $x \in C$ is a geometric point with image $s \in S$, then there are \'etale neighborhoods $V$ of $x$ and $U$ of $s$ so that $V \to U$ has a strict morphism to an \'etale-local model
$V' \to U'$ where $V' \to U'$ is one of the following:
\begin{enumerate}
  \item (the smooth germ) $V' = \mathbb{A}^1_{U'} \to U'$ and the log structure on $V'$ is pulled back from the base;
  \item (the germ of a marked point) $V' = \mathbb{A}^1_{U'} \to U'$ with the log structure pulled back from the toric log structure on $\mathbb{A}^1$;
  \item (the node) $V' = \underline{\Spec} \mathscr{O}_{U'}[x,y] / (xy - t)$ for $t \in \mathscr{O}_{U'}$. The log structure on $V$ is pulled back from the multiplication map $\mathbb{A}^2 \to \mathbb{A}^1$ of toric varieties along the morphism
  $U' \to \mathbb{A}^1$ of logarithmic schemes induced by $t$.
\end{enumerate}
\end{theorem}

The tropicalization of a log curve is the dual graph of its underlying nodal curve, enriched with the data of the smoothing parameters from its log structure.

\begin{definition}
Given an $n$-marked log curve $\pi: C \to S$ where $S$ is a log point, the \emph{tropicalization} $\trop(C)$ of $C$ is the $n$-marked tropical curve with
edge lengths in $\Gamma(S, \ol{M}_{S})$ which has
\begin{enumerate}
  \item a vertex for each component of $C$;
  \item an edge for each node of $C$, incident to the components of $C$ which form the branches of the node;
  \item a leg for each marked point of $C$, rooted at the component of $C$ to which the marked point belongs
\end{enumerate}
and
\begin{enumerate}
  \item for each vertex $v$, the genus $g(v)$ is the genus of the normalization of the corresponding component of $C$;
  \item for each edge $e$, the length $\delta(e) \in \Gamma(S, \ol{M}_S)$ is the smoothing parameter of the node $e$.
\end{enumerate}
\end{definition}

See Figure \ref{fig:dual_graph_example} for an example of tropicalization. Note that the tropicalization may contain loops: consider the nodal cubic.

Sections of the characteristic sheaf of $C$ are interpreted tropically as piecewise linear functions.

\begin{theorem}
Let $\pi : C \to S$ be a log curve over the spectrum of an algebraically closed field. Then there is a bijection
\begin{align*}
  \PL : \Gamma(C, \ol{M}_C) &\overset{\sim}{\longrightarrow} \PL(\trop(C)) \\
     \sigma &\mapsto \PL(\sigma)
\end{align*}
where
\begin{enumerate}
  \item the value of $\PL(\sigma)$ at a vertex $v$ of $\Gamma(C)$ is the stalk of $\sigma$ at the generic point of the corresponding component of $C$;
  \item the slope of $\PL(\sigma)$ at a leg $l$ of $\Gamma(C)$ is the image of $\sigma$ in $(\ol{M}_C/\pi^{-1}\ol{M}_S)_{p} \cong \mathbb{N}$ where $p$ is the marked point corresponding to $l$.
\end{enumerate}
\end{theorem}
\begin{proof}
See, for example \cite[Remark 7.3]{ccuw}.
\end{proof}

For a general log curve, this interpretation extends nicely over an \'etale neighborhood of each point.

\begin{theorem} \label{thm:uniform_charts}
Let $\pi : C \to S$ be a log curve and let $s$ be a geometric point of $S$. Then there is an \'etale neighborhood $U$ of $s$ in $S$ so that
\begin{enumerate}
  \item $\Gamma(U, \ol{M}_S) \to \ol{M}_{S,s}$ and $\Gamma(C|_U, \ol{M}_{C}) \to \Gamma(C|_s, \ol{M}_{C|_s})$  are isomorphisms;
  \item for each geometric point $t$ of $U$, there is a canonical face contraction
  \[
    \trop(C_s) \to \trop(C_t)
  \]
  induced by
  \[
    \ol{M}_{S,s} \overset{\sim}{\longleftarrow} \Gamma(U, \ol{M}_S) \longrightarrow \ol{M}_{S,t}.
  \]
  Moreover, this face contraction respects associated piecewise linear functions in the sense that
  \[
    \xymatrix{
      \Gamma(C|_s, \ol{M}_{C|_s}) \ar[d]_{\PL} & \Gamma(C|_U, \ol{M}_C) \ar[l]_{\sim} \ar[r] & \Gamma(C|_t, \ol{M}_{C|_t}) \ar[d]^{\PL} \\
      \PL(\trop(C|_s)) \ar[rr] & & \PL(\trop(C|_t))
    }
  \]
  commutes.
\end{enumerate}
\end{theorem}
\begin{proof}
This follows, for example, from the existence of ``uniform sets of charts,'' constructed in \cite[Proposition 2.3.13]{bozlee_thesis}.
\end{proof}

It follows that to define a section of the characteristic sheaf of $C$ it is equivalent to specify a piecewise linear function on each geometric fiber of $C$ so that the resulting piecewise linear functions are compatible with generization.

\begin{definition}
An $n$-marked log curve is a log curve $\pi : C \to S$ equipped with disjoint sections $\sigma_1, \ldots, \sigma_n : S \to C$ with image the marked points of $C$. An $n$-marked log curve is \emph{stable} if its underlying family of marked
nodal curves is Deligne-Mumford stable.
\end{definition}

\begin{theorem}(\cite[Theorem 4.5]{fkato_deformations})
There is a log structure on $\ol{\mathcal{M}}_{g,n}$, called the \emph{basic log structure}, so that $\ol{\mathcal{M}}_{g,n}$ represents
the stack of stable $n$-marked log curves of genus $g$ over the category of fs log schemes.

\end{theorem}

We will generally regard $\ol{\mathcal{M}}_{g,n}$ as a stack with the basic log structure and will freely confuse $\ol{\mathcal{M}}_{g,n}$ with its underlying algebraic stack. A stable log curve $\pi : C \to S$ is said to have the $\emph{basic log structure}$ if its log structure is pulled back from that of the universal stable log curve of $\ol{\mathcal{M}}_{g,n}$. In the case that $S$ is a geometric point, $\pi$ has the basic log structure if and only if the characteristic monoid $\ol{M}_S$ is freely generated by the edge lengths of $\trop(C)$. 

\subsection{Radially aligned curves}

We now build up the terminology to work with radially aligned curves. These were introduced by Santos-Parker in \cite{keli_thesis} under the name of ordered log curves and then popularized in \cite{rsw}.

\begin{definition}
Let $\Gamma$ be a tropical curve. A \emph{path} $W$ in $\Gamma$ is a sequence $v_0e_1v_1e_2 \cdots e_kv_k$ of vertices and edges in $\Gamma$ so that the vertices $v_i$ are distinct and $v_{i - 1}$ and $v_i$ are the ends of the edge $e_i$ for all $i$. Given subsets $A$ and $B$ of $V(\Gamma)$, we say that $W$ is a path from $A$ to $B$ if $v_0 \in A$, $v_k \in B$, and $v_i \not\in A \cup B$ for $i \neq 0, k$.
\end{definition}

\begin{definition}
Given a proper curve $C$ over the spectrum of an algebraically closed field, a \emph{subcurve} of $C$ is a union of irreducible components of $C$, possibly empty.

The \emph{core} of $C$ is the minimal connected subcurve of $C$ with the same genus as $C$. Analogously the core of a tropical curve $\Gamma$ is the minimal connected vertex-induced subgraph of the same genus as $\Gamma$.
\end{definition}

\begin{definition}
Given a tropical curve $\Gamma$ of genus one, we define a piecewise linear function $\lambda$ on $\Gamma$ measuring ``distance from the core'' as follows. If $v$ is a vertex in the core of $\Gamma$, we set
\[
  \lambda(v) = 0.
\]
If $v$ is a vertex outside of the core of $\Gamma$, we let $W = v_0e_1v_1e_2 \cdots e_kv_k$ be the unique path from the core of $\Gamma$ to $v$ and set
\[
  \lambda(v) = \sum_{i = 1}^k \delta_{e_i}.
\]
Finally, we set the slope of $\lambda$ to be 1 at all marked points.

This is compatible with generization, so for any stable log curve $(\pi : C \to S; \sigma_1, \ldots, \sigma_n)$ of genus one, we let $\lambda \in \Gamma(S, \ol{M}_S)$ be the unique section of the characteristic bundle whose
restriction to geometric fibers has corresponding piecewise linear function as in the last paragraph.
\end{definition}

\begin{definition} \label{def:msbar_order}
If $P$ is any fs sharp monoid, we give the elements of $P$ a partial order by the rule $p \leq q$ if and only if there exists $r \in P$ with $q = p + r$.
\end{definition}

\begin{definition} \label{def:basic_radially_aligned}
A stable $n$-marked tropical curve of genus one with edge lengths in $P$ is \emph{radially aligned} if, for each pair of vertices $v, w$ of $\Gamma$, $\lambda(v)$ is comparable to $\lambda(w)$ in $P$.

Given such a radially aligned curve, let
\[
  0 < \rho_1 < \cdots < \rho_k
\]
be the distinct values of $\lambda(v)$ as $v$ varies over the components of $C$, and let $\delta_1, \ldots, \delta_l$ be the lengths of the edges of $\trop(C)$ internal to the core of $\Gamma$. Let $e_1 = \rho_1, e_2 = \rho_2 - \rho_1, \ldots, e_k = \rho_k - \rho_{k - 1}$. If $P$ is freely generated by
\[
  \{ e_1, \ldots, e_k \} \cup \{ \delta_1, \ldots, \delta_l \},
\]
then we say that $\Gamma$ is a \emph{basic radially aligned tropical curve}. An element of $P$ is said to \emph{have no contribution from the core} if it lies in the submonoid generated by $e_1, \ldots, e_k$.

A stable log curve $(\pi : C \to S; \sigma_1, \ldots, \sigma_n)$ of genus one with $n$ markings is \emph{radially aligned} or \emph{has a basic radially aligned log structure} if the tropicalizations of its geometric fibers
with their pulled back log structure are respectively radially aligned or basic radially aligned. An element $\rho \in \Gamma(S, \ol{M}_{S})$ \emph{has no contribution from the core} if the same holds of its
stalks at the geometric points of $S$.
\end{definition}

There is a moduli stack with log structure parametrizing radially aligned log curves.

\begin{theorem} \label{thm:M_1n_rad_exists} \cite[Proposition 3.3.4]{rsw}\hfill
\begin{enumerate}
\item There is a Deligne-Mumford stack with locally free log structure $\ol{\mathcal{M}}_{1,n}^{rad}$ whose $S$-points for $S$ an fs log scheme are the $n$-marked radially aligned curves $\pi : C \to S$ over $S$.  We say its log structure is \emph{the} basic radially aligned log structure.

\item There is a natural map $\ol{\mathcal{M}}_{1,n}^{rad} \to \ol{\mathcal{M}}_{1,n}$ induced by a logarithmic blowup and it restricts to an isomorphism on $\mathcal{M}_{1,n}$.
\end{enumerate}
\end{theorem}

A stable log curve $(\pi : C \to S; \sigma_1, \ldots, \sigma_n)$ has a basic radially aligned
log structure precisely when the log structure on $\pi : C \to S$ is that pulled back from the
the universal curve $\mathcal{C}_{1,n}^{rad} \to \ol{\mathcal{M}}_{1,n}^{rad}$ along the map
$S \to \ol{\mathcal{M}}_{1,n}^{rad}$. We remark that a fixed family of nodal curves may be
enhanced to a family of basic radially-aligned log curves in more than one way, which we
illustrate with an example.

\begin{example}
Let $\pi : C \to S = \Spec k$ be a stable curve with the basic log structure over the
spectrum of an algebraically closed field, and suppose that its tropicalization is

\medskip

\begin{center}
  \includegraphics[width=1.5in]{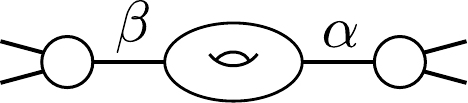}
\end{center}

\medskip

There is an associated map $S \to \ol{\mathcal{M}}_{1,4}$. The basic log structure on $S$ comes from the chart
\[
  \N \tilde{\alpha} \oplus \N \tilde{\beta} \to k
\]
sending $\tilde{\alpha}, \tilde{\beta} \mapsto 0$. The edge lengths $\alpha, \beta$
are the respective images of $\tilde{\alpha}$ and $\tilde{\beta}$ in the characteristic sheaf. Locally in $\ol{\mathcal{M}}_{1,4}$ near the image of $S$,
the map $\ol{\mathcal{M}}_{1,4}^{rad} \to \ol{\mathcal{M}}_{1,4}$ is given by the log blowup
of the log ideal generated by $\alpha, \beta$, as these are the distances that we wish to make comparable. We refer the interested reader to \cite[Chapter III, Section 2.6]{ogus} for details
on log blowups.

We may compute all of the basic radially aligned log stuctures on $C$ by computing the restriction of this blowup to $S$. We construct the blowup by first freely adjoining the element $\tilde{\alpha} - \tilde{\beta}$ to the log structure of $S$,
adjoining an element to $k$ for $\tilde{\alpha} - \tilde{\beta}$ to map to, doing likewise for $\tilde{\beta} - \tilde{\alpha}$, and finally gluing over the overlap. That is, $S \times_{\ol{\mathcal{M}}_{1,4}} \ol{\mathcal{M}}_{1,4}^{rad}$ possesses a cover by two open sets $U = \Spec k[t]$ (where $\alpha \geq \beta)$ and $V = \Spec k[t^{-1}]$ (where $\beta \geq \alpha$) with log structure on $U$ induced by
\begin{align*}
  \N (\tilde{\alpha} - \tilde{\beta}) \oplus \N \tilde{\beta} &\to k[t] \\
  \tilde{\alpha} - \tilde{\beta} & \mapsto t \\
  \tilde{\beta} & \mapsto 0
\end{align*}
and log structure on $V$ induced by
\begin{align*}
  \N \tilde{\alpha} \oplus \N (\tilde{\beta} - \tilde{\alpha}) &\to k[t^{-1}] \\
  \tilde{\alpha} & \mapsto 0 \\
  \tilde{\beta} - \tilde{\alpha} & \mapsto t^{-1}.
\end{align*}
The two charts are glued in the obvious way. Notice that on the intersection
$U \cap V = \Spec k[t,t^{-1}]$ the sections $\tilde{\alpha} - \tilde{\beta}$ and $\tilde{\beta} - \tilde{\alpha}$ of the log structure restrict to units, so that their images in
the characteristic sheaf are 0. It follows that $\alpha$
and $\beta$ are equal over a $\mathbb{G}_m$'s worth of possible basic radially aligned enhancements of $C$. See Figure \ref{fig:radial_alignments_example}.

\begin{figure}
\begin{center}
  \includegraphics[height=2in]{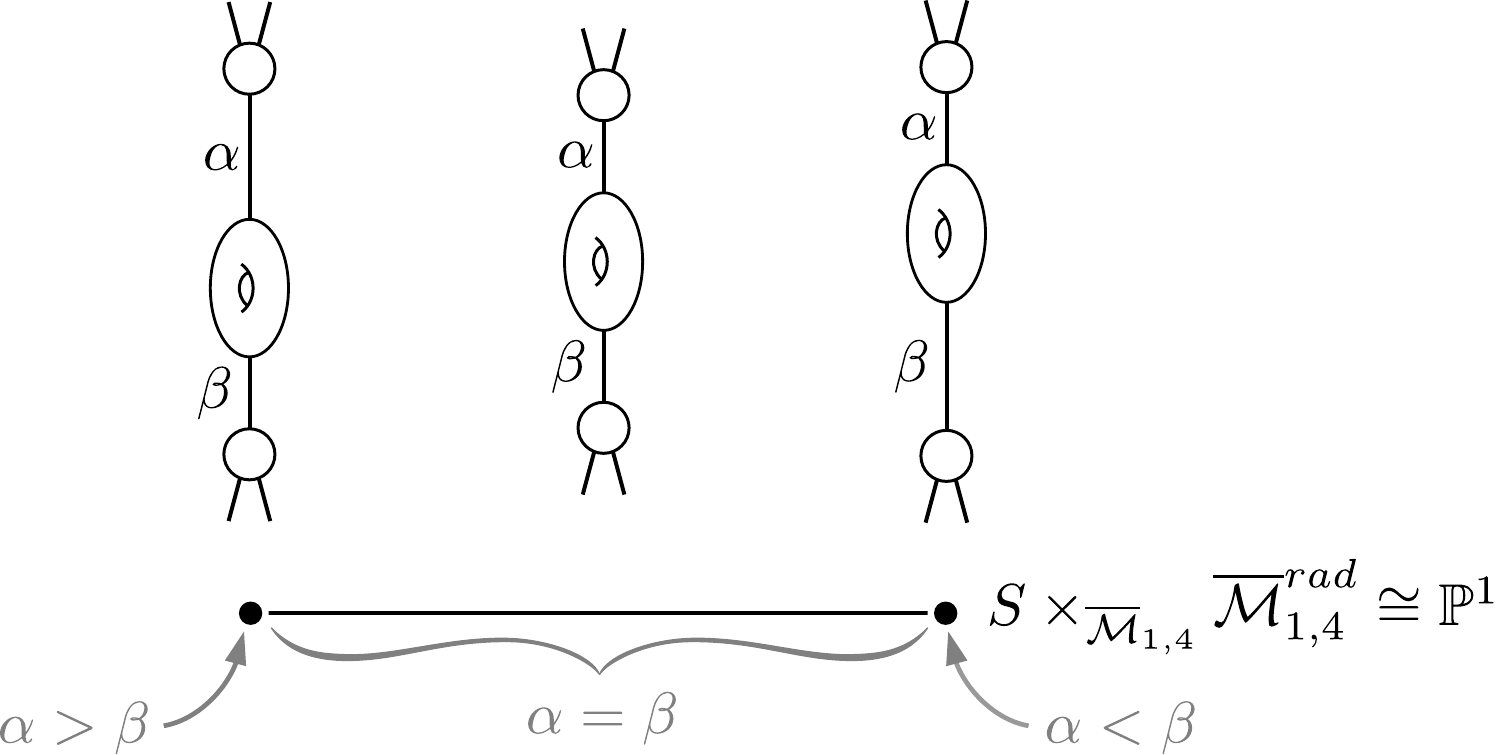}
\end{center}
\caption{The $\mathbb{P}^1$ of radial alignments of a basic stable log curve with two edges.}
\label{fig:radial_alignments_example}
\end{figure}

\end{example}

\begin{definition} \label{def:partition_type}
Let $\Gamma$ be a radially aligned tropical curve with ordered radii $0 < \rho_1 < \ldots < \rho_k$.

Given a radius $\rho$, we may form a tropical curve $\widetilde{\Gamma}(\rho)$ by subdividing the edges and legs of $\Gamma$ where $\lambda = \rho$,
then deleting the locus where $\lambda < \rho$. We define the \emph{partition associated to the radius $\rho$} to be the partition of $\{1, \ldots, n\}$ induced by the components of $\widetilde{\Gamma}(\rho)$, and we denote it by $\Part(\rho)$.

We say that the resulting strict chain of partitions
\[
  \Part(\rho_1) \prec \Part(\rho_2) \prec \cdots \prec \Part(\rho_k)
\]
is the \emph{partition type} of $\Gamma$. See Figure \ref{fig:partition_type} for an example.
\end{definition}

It would be natural to include the partition $\Part(0)$ in the partition type as well, but we choose not to for a few reasons. The first is that we always have $\Part(0) = \{ \{ 1, 2, \ldots, n \} \}$, since $\Part(0)$ is the partition of the markings induced by deleting no components. So including $\Part(0)$ in the list does not convey more information. For another, unlike the other comparisons, the comparison $\Part(0) \preceq \Part(\rho_1)$ need not be strict: it may be that both are the indiscrete partition. For example, see Figure \ref{fig:partition_type}.

\begin{figure}
    \centering
    \includegraphics[width=2in]{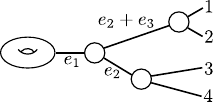}
    \caption{A basic radially aligned curve with partition type
    \[
      \{ \{ 1, 2, 3, 4 \} \} \prec \{ \{1, 2 \}, \{3, 4\} \} \prec \{ \{ 1, 2 \}, \{ 3 \}, \{4\} \}.
    \]
    We draw a torus to indicate a vertex of genus one.}
    \label{fig:partition_type}
\end{figure}

\section{Contractions of the universal radially-aligned curve} \label{sec:radially_aligned_contractions}

Part of the utility of families of radially aligned curves is that they are easy to contract to families of curves with Gorenstein singularities, even at the level of
a universal curve. By exploring the possible contractions of the universal curve of $\ol{\mathcal{M}}_{1,n}^{rad}$ we find regular birational maps
$\ol{\mathcal{M}}_{1,n}^{rad} \to \ol{\mathcal{M}}_{1,n}(Q)$ for each $Q$. It was this computation that identified the $Q$-stable moduli spaces.

The following theorem says that in order to contract a family of radially aligned curves, all we need is the data of a tropical radius for
each curve in the family. This idea was key for the results of \cite{keli_thesis} and \cite{rsw} and can be done using their language, see
\cite[Section 5]{keli_thesis} and \cite[Section 3.7]{rsw}. We give a proof using the language of \cite{bozlee_thesis} for its convenience and generality.

\begin{theorem} \label{thm:contraction}
Let $\pi : C \to S$ be a family of $n$-marked radially aligned log curves. Let $\rho \in \Gamma(S, \ol{M}_S)$ be a section of the
characteristic monoid such that ``$\rho$ is a radius at all geometric points of $s$", that is, for each geometric point $s$ of $S$,
there is a vertex $v$ of $\trop(C|_s)$ so that $\lambda(v) = \rho|_s$.

Then there is a diagram
\[
\xymatrix{
    & \widetilde{C} \ar[rd]^{\tau} \ar[dl]_{\phi} & \\
   C \ar[rd]_{\pi} & & \ol{C} \ar[dl]^{\ol{\pi}} \\
  & S &
}
\]
where
\begin{enumerate}
  \item $\phi$ is a log blowup inducing the subdivision at the locus where $\lambda = \rho$ on tropicalizations;
  \item $\ol{\pi} : \ol{C} \to S$ is a flat and proper family of Gorenstein curves of genus one;
  \item $\tau$ is a surjective map whose restriction to geometric fibers contracts the locus (if non-empty) where $\lambda < \rho$
  to an elliptic singularity of level $\Part(\rho|_s)$ and restricts further to an isomorphism in the complement of this locus.
\end{enumerate}

Moreover, formation of the diagram commutes with base change in $S$.
\end{theorem}

\begin{proof}
The construction of $\phi$ is standard. Using the language of \cite{bozlee_thesis}, we then define a mesa $\ol{\lambda} \in \Gamma(\tilde{C}, \ol{M}_{\tilde{C}})$ on the resulting family of log curves $\widetilde{C} \to S$ with the formula
\[
  \ol{\lambda} = \max \{ \rho - \lambda, 0 \}.
\]
It is easy to check that $\ol{\lambda}$ defines a steep mesa with support on the locus where $\lambda < \rho$, so the main theorem of \cite{bozlee_thesis} yields the claimed diagram with the required properties. To see that the
elliptic singularities of $\ol{C}$ have the claimed level, we note that the branches of the singularity will be the images of the connected components of the locus in $\widetilde{C}$ where $\lambda \geq \rho$. These are precisely the connected components considered in the definition
of $\Part(\rho)$.
\end{proof}

We want to apply this theorem to the universal curve of $\ol{\mathcal{M}}_{1,n}^{rad}$. Our next task is to reduce the problem of enumerating
the possible radii $\rho \in \Gamma(\ol{\mathcal{M}}_{1,n}, \ol{M}_{\ol{\mathcal{M}}_{1,n}^{rad}})$ to something manageable.

\begin{lemma} \label{lem:basic_tropical_auts}
Let $\Gamma$ be a basic radially aligned curve. Then:
\begin{enumerate}
  \item If the core of $\Gamma$ consists of a vertex with a self-loop, then the only non-trivial automorphism of $\Gamma$ is the automorphism reversing the loop, but identity on everything else.
  \item If the core of $\Gamma$ consists of a pair of vertices with two edges, the only non-trivial automorphism of $\Gamma$ is the automorphism exchanging the edges of the core, and the identity on everything else.
  \item Otherwise, $\Gamma$ has no non-identity automorphisms.
\end{enumerate}
\end{lemma}

\begin{proof}
Suppose $\Gamma$ is a basic radially aligned tropical curve. Let $\phi : \Gamma \to \Gamma$ be an invertible weighted edge contraction.

We argue that $\phi$ is the identity on the vertices of $\Gamma$. Let $v$ be a vertex of $\Gamma$. Notice that the complement of the core of $\Gamma$ consists a forest of trees, each of which we can root
at the vertex that attaches to the core. Furthermore, due to stability,
\begin{enumerate}
  \item if $v$ is a vertex outside of the core of $\Gamma$, $v$ is uniquely identified by the markings that lie on $v$ and the descendants of $v$;
  \item if $v$ is a vertex inside the core of $\Gamma$, there is at least one tree attached to $v$, and those trees are uniquely identified by their markings.
\end{enumerate}
An automorphism of $\Gamma$ must in particular preserve the markings. Therefore, $\phi$ fixes all of the vertices of $\Gamma$. This implies that $\phi$ fixes all edges except possibly those who share incident vertices. This implies the result.
\end{proof}

\begin{remark}
For a fixed $n$, there are only finitely many isomorphism classes of $n$-marked tropical curve with the basic radially aligned log structure.
\end{remark}

\begin{lemma} \label{lem:radially_aligned_strata}
Let $\Gamma$ be an $n$-marked basic radially aligned tropical curve with edge lengths in $P$. Write $\mathcal{M}^{rad}_{\Gamma}$ for the substack of $\ol{\mathcal{M}}_{1,n}^{rad}$ of curves whose tropicalizations are isomorphic to $\Gamma$. Then $\mathcal{M}^{rad}_{\Gamma}$ is an irreducible locally closed substack of $\ol{\mathcal{M}}_{1,n}^{rad}$.
\end{lemma}

\begin{proof}
Write $\mathcal{M}_{\Gamma}$ for the locally closed substack of stable $n$-marked curves of genus one whose dual graph is the underlying graph of $\Gamma$.
Recall that $\mathcal{M}_{\Gamma} \cong \prod_{v \in V(\Gamma)} \mathcal{M}_{g(v), \mathrm{val}(v)}$, where the valence $\mathrm{val}(v)$ is the number of flags incident to $v$. Since the $\mathcal{M}_{g(v), \mathrm{val}(v)}$'s are geometrically irreducible, so is $\mathcal{M}_{\Gamma}$. There is a forgetful map $\mathcal{M}^{rad}_{\Gamma} \to \mathcal{M}_{\Gamma}$ given by forgetting the log structure.

We recall from \cite[Proposition 3.3.4]{rsw} that the map $\ol{\mathcal{M}}^{rad}_{1,n} \to \ol{\mathcal{M}}_{1,n}$ is locally given as follows. Suppose given a map $S \to \ol{\mathcal{M}}_{1,n}$ so $S$ admits a global chart by a monoid $Q$. This induces a map $S \to V = \Spec \mathbb{Z}[Q]$. Let $\sigma$ be the rational polyhedral cone dual to $Q$. Let $\Sigma$ be the fan obtained by subdividing $\sigma$ along the hyperplanes where $\lambda(v) = \lambda(w)$ as $v$ and $w$ range among the vertices of $\trop(C_s)$,
and let $W$ be the toric variety associated to $\Sigma$. Then
\[
  S \times_{\ol{\mathcal{M}}_{1,n}} \ol{\mathcal{M}}_{1,n}^{rad} \cong S \times_V W.
\]

Suppose that $S \to \mathcal{M}_{1,n}^{rad}$ factors through $\mathcal{M}^{rad}_{\Gamma}$. Then $Q$ is the free monoid on the edges of $\Gamma$. By construction, there is a
cone of $\Sigma$ associated to each possible choice of ordering of the
distances $\lambda(v)$ as $v$ varies over $V(\Gamma)$. Since $\Gamma$ is basic radially
aligned, these distances are ordered, and their ordering
determines a cone $\tau$ of $\Sigma$. Let $W_{\tau} \subseteq W$ be the torus orbit associated to $\tau$. One may check that the locus in $S \times_V W$ in which the tropicalization is
isomorphic to $\Gamma$ is precisely the locus $S \times_V W_{\tau}$: this is the locus in which the stalks of the characteristic sheaf agree with $P$. Letting $S$ vary over a smooth cover of $\mathcal{M}^{rad}_{\Gamma}$, we see that $\mathcal{M}^{rad}_{\Gamma} \to \mathcal{M}_{\Gamma}$ is smooth-locally a $W_{\tau}$-bundle. Since the target is irreducible and $W_\tau$ is irreducible, $\mathcal{M}^{rad}_{\Gamma}$ is irreducible. Moreover it is a locally closed substack of $\ol{\mathcal{M}}_{1,n}^{rad}$ as this is true of $W_{\tau}$ inside $W$.
\end{proof}

\begin{lemma} \label{lem:characteristic_sections_as_tropical_lengths}
Let $I$ be the set of isomorphism classes of $n$-marked basic radially aligned tropical curves. Fix a representative $\Gamma$ with edge lengths in $P_\Gamma$ for each isomorphism class.

To give a section $\rho$ of the characteristic sheaf of $\ol{\mathcal{M}}_{1,n}^{rad}$ with no contribution from the core, it is equivalent to specify for each isomorphism class of basic radially aligned tropical curve an element $\rho_\Gamma \in P_\Gamma$ with no contribution from the core such that whenever $\Gamma, \Gamma' \in I$ and $\pi : \Gamma \to \Gamma'$ is a face contraction with $\pi^\sharp : P_\Gamma \to P_{\Gamma'}$ a quotient map, then $\pi^*\rho_\Gamma = \rho_{\Gamma'}$
\end{lemma}
\begin{proof}

For brevity, write $\ol{M}$ for the characteristic sheaf of $\ol{\mathcal{M}}_{1,n}^{rad}$. Since $\ol{\mathcal{M}}_{1,n}^{rad}$ is a Deligne-Mumford stack, we may choose an \'etale cover $U \to \ol{\mathcal{M}}_{1,n}^{rad}$ by a scheme $U$. For each point
$x$ of $U$, choose an algebraic closure $\ol{k}(x)$ of the residue field $k(x)$, and write $\ol{x}^U : \Spec \ol{k}(x) \to U$ for the natural map to $U$ and $\ol{x}$ for its composite with $\ol{\mathcal{M}}_{1,n}^{rad}$. Write $\Gamma_{\ol{x}}$ for
the tropicalization of the basic radially aligned log curve associated to $\ol{x}$ and write $P_{\ol{x}}$ for its associated monoid, i.e., $\ol{M}_{\ol{x}}$.
 
We may apply Theorem \ref{thm:uniform_charts} to find an \'etale neighborhood $U_{\ol{x}}$ of $\ol{x}^U$ in $U$ over which the pullback of the universal curve of $\ol{\mathcal{M}}_{1,n}^{rad}$ has the properties of
the theorem. Since $U$ is an \'etale cover of $\ol{\mathcal{M}}_{1,n}^{rad}$ and the log structure on $\ol{\mathcal{M}}_{1,n}^{rad}$ is divisorial, the maps $P_{\ol{x}} \overset{\sim}{\longleftarrow} \Gamma(U_{\ol{x}}, \ol{M}) \to \ol{M}_t$ vary
through \emph{all} of the quotients of $P_{\ol{x}}$ by its generators as $t$ varies through the geometric points of $U_{\ol{x}}$.
Then, since the $U_{\ol{x}}$'s form a cover, we may identify the global sections of $\ol{M}$ with elements $(\rho_{\ol{x}})$ of $\prod_{\ol{x}} \Gamma(U_{\ol{x}}, \ol{M}) \cong \prod_{\ol{x}} P_{\ol{x}}$, suitably
compatible on overlaps.

Now, for a pair of points $\ol{x}$ and $\ol{y}$, sections $\rho_{\ol{x}} \in \Gamma(U_{\ol{x}}, \ol{M})$ and $\rho_{\ol{y}} \in \Gamma(U_{\ol{y}}, \ol{M})$ agree on $U_{\ol{x},\ol{y}} := U_{\ol{x}} \times_{\ol{\mathcal{M}}_{1,n}^{rad}} U_{\ol{y}}$ if and only if their stalks at
geometric points $\ol{z}$ of $U_{\ol{x}, \ol{y}}$ agree. This translates to the statement that whenever $\Gamma_{\ol{z}}$ is a face contraction of both $\Gamma_{\ol{x}}$ and $\Gamma_{\ol{y}}$, and $\phi : \Gamma_{\ol{z}} \to \Gamma_{\ol{z}}$
is an automorphism, the stalk of $\rho_{\ol{x}}$ at $\ol{z}$ is $\phi^{\sharp}$ applied to the stalk of $\rho_{\ol{y}}$ at $\ol{z}$.

Suppose $\ol{x}$ and $\ol{y}$ are points, $\Gamma \in I$ and $\Gamma_{\ol{x}} \cong \Gamma \cong \Gamma_{\ol{y}}$. Then, in the notation of the previous lemma, since $\mathcal{M}^{rad}_\Gamma$ is irreducible, $U_{\ol{x},\ol{y}}$ must contain a point $\ol{z}$ that also maps into $\mathcal{M}^{rad}_\Gamma$. Then the elements corresponding to $\rho_{\ol{x}}$ and $\rho_{\ol{y}}$ on $P_\Gamma$ must differ by at most an automorphism of $\Gamma$. By Lemma \ref{lem:basic_tropical_auts}, they are actually equal, so we
have a well-defined element $\rho_\Gamma$ of $P_\Gamma$. The agreement of stalks at other points implies that the $\rho_\Gamma$'s are compatible with edge contraction. We obtain the converse by reversing this construction.
\end{proof}

In view of Lemma \ref{lem:characteristic_sections_as_tropical_lengths}, we introduce the notion of \emph{universal radius}.

\begin{definition} \label{def:univ_radius}
An $n$-marked \emph{universal radius} consists of the data of an element $\rho_\Gamma \in P_\Gamma$ for each $n$-marked basic radially aligned curve $\Gamma$ so that:
\begin{enumerate}
  \item for each $\Gamma$, $\rho_\Gamma = \lambda(v)$ for some vertex $v$ of $\Gamma$;
  \item if $\Gamma, \Gamma'$ are two $n$-marked radially aligned curves and $\pi : \Gamma \to \Gamma'$ is a face contraction, then $\pi^*\rho_\Gamma = \rho_{\Gamma'}$.
\end{enumerate}

We use the shorthand notation $(\rho_\Gamma)$ for the tuple of radii making up a universal radius, and denote by $\mathfrak{R}^{uni}_n$ the set of $n$-marked
universal radii.
\end{definition}

\begin{remark}
Condition (i) implies each $\rho_\Gamma$ has no contribution from the core. Condition (ii) implies that we only need to keep track of the \emph{finite} data of a choice of radius $\rho_\Gamma$ for each isomorphism class of $n$-marked basic radially aligned curve.
The maps $P_\Gamma \to P_{\Gamma'}$ induced by face contractions are just the coordinate projections. Therefore all we have to worry about to
satisfy condition (ii) is what happens when we set various subsets of the generators of $P_\Gamma$ ($\delta_1, \ldots, \delta_l$ and
$e_1, \ldots, e_k$ in the notation of Definition \ref{def:basic_radially_aligned}) to zero.
\end{remark}

We have therefore reduced the problem of finding a section of the characteristic sheaf of $\ol{\mathcal{M}}_{1,n}^{rad}$ to giving the finite collection of
tropical data that make up a universal radius. This is still a fair amount of data: see Figure \ref{fig:universal_radius_example} for an example when $n = 3$.
We will see that we can reduce the data of a universal radius further to that of a downward closed subset $Q$ of partitions on
$\{1, \ldots, n\}$. In Figure \ref{fig:universal_radius_example}, for example, the corresponding downward closed subset will be
\[
  Q = \{ \, \, \{ \{ 1, 2, 3 \} \}, \quad \{ \{2\}, \{ 1, 3 \} \}, \quad \{ \{3\}, \{1, 2\} \}  \, \, \}.
\]
This can be read off from the second row of the figure.

\begin{figure} 
    \centering
    \includegraphics[width=6in]{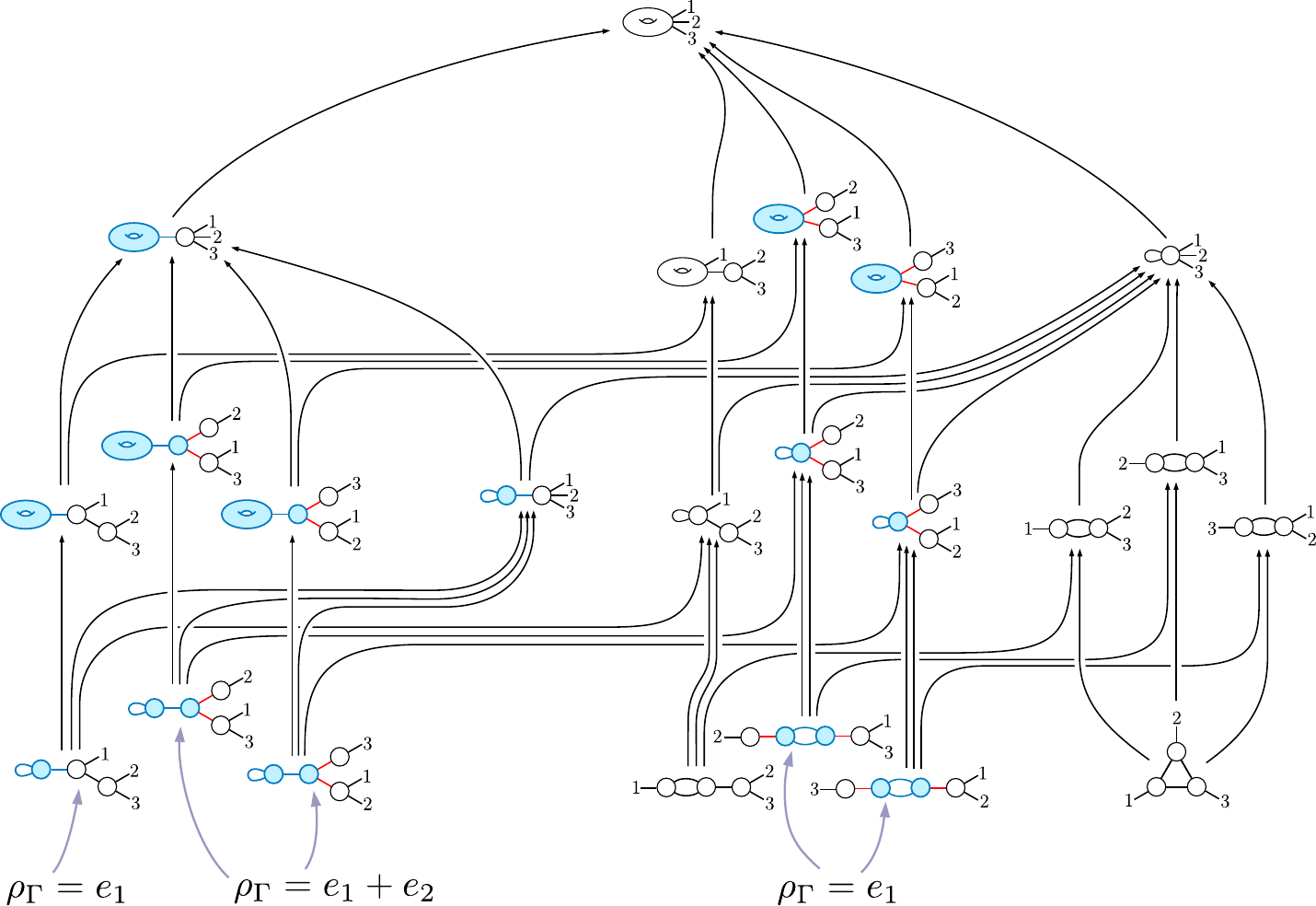}
    \caption{An example of the subdivided universal curve $\widetilde{C}$ on $\ol{\mathcal{M}}_{1,3}^{rad}$ associated to a universal radius. A torus indicates a vertex of genus one. We have labelled the nonzero $\rho_\Gamma$'s on the most degenerate tropical curves using the notation of Definition \ref{def:basic_radially_aligned}. The radii on other tropical curves can be deduced by following the indicated face contractions. The red edges of a particular curve have equal length: they come about by subdividing at the radius $\rho_\Gamma$. The blue components are those to be contracted. }
    \label{fig:universal_radius_example}
\end{figure}

\begin{definition}
If $\Gamma$ is an $n$-marked basic radially aligned tropical curve with $k$ nonzero radii $0 < \rho_1 < \cdots < \rho_k$ and smooth core, we say that $\Gamma$ is a \emph{$k$-layer tree}.
\end{definition}

\begin{lemma} \label{lem:partitions_increase_strictly}
Let $\Gamma$ be a basic radially-aligned tropical curve with $k$ nonzero radii $0 < \rho_1 < \cdots < \rho_k$. Then for each index $i < k$, there is a strict refinement $\Part(\rho_i) \prec \Part(\rho_{i+1})$.
\end{lemma}

\begin{proof}
Notice that if $\Gamma$ is replaced by the tropical curve in which all edges of the core are contracted, the sequence of partitions remains the same. Therefore we may assume that $\Gamma$ is a $k$-layer tree.

Orient the tree $\Gamma$ by taking its core as a root.
Resuming the notation of Definition \ref{def:partition_type}, let $V_i$ be the set of roots of the forest $\widetilde{\Gamma}(\rho_i)$ for each $i$. Observe that for each $i$, there are bijections between $V_i$, the connected components of $\widetilde{\Gamma}(\rho_i)$, and the parts of the partition $\Part(\rho_i)$. Notice that the connected components of $\widetilde{\Gamma}(\rho_{i+1})$ factor through the connected components
of $\widetilde{\Gamma}(\rho_i)$, so the partition of $\{1,\ldots,n\}$ induced by the connected components of $\widetilde{\Gamma}(\rho_{i+1})$
refines that induced by the components of $\widetilde{\Gamma}(\rho_{i})$.

To see that the refinement is strict, let $v$ be a vertex of $\Gamma$ so that $\lambda(v) = \rho_i$. By stability, there must be at least two
flags leaving $v$ in the direction of increasing $\lambda$. Then there are at least two vertices of $V_{i+1}$ that belong to the component of $\Gamma(\rho_i)$ containing $v$. It follows that the refinement is strict.
\end{proof}

\begin{proposition}
$1$-layer trees are in bijection with the non-discrete partitions of $\{1,\ldots, n\}$. 
\end{proposition}

\begin{proof}
Given a non-discrete partition $p$ of $\{1,\ldots, n\}$, we can construct a $1$ layer tree $\Gamma(p)$ as follows. That is, let $p = \{p_1, \ldots, p_r\}$. Start with a genus $1$ vertex $v$ then for each $1 \leq i \leq r$ attach a vertex $v_i$ to $v$ such that i) $v_i$ is distance $\rho_1$ from $v$ and ii) the elements of $p_i$ are precisely the legs attached to $v_i$. After stabilizing, we obtain $\Gamma(p)$.

If $\Gamma$ is a $1$ layer tree, then $\Gamma \mapsto \Part(\rho_1)$ gives a map from $1$ layer trees to non-discrete partitions. The maps $p \mapsto \Gamma(p)$ and $\Gamma \mapsto \Part(\rho_1)$ are inverses. 
\end{proof}

\begin{definition}
Let $\Gamma$  be a radially aligned tropical curve with monoid
\[ 
\mathbb{N}e_1 \oplus \cdots  \oplus \mathbb{N} e_k \oplus \mathbb{N} \delta_1 \oplus \cdots \oplus \mathbb{N} \delta_l
\]
and let $\Gamma_{e_i}$ denote the tropical curve corresponding the monoid map $P \to \mathbb{N}$ taking $e_i \mapsto 1$, $e_j \mapsto 0$ for all $j \neq i$, and $\delta_j \mapsto 0$ for all $j$. We define $\Gamma_{\delta_i}$ similarly; it is the tropical curve corresponding to the monoid map $P \to \mathbb{N}$ taking $\delta_i \mapsto 1$ and all other generators to $0$. 
\end{definition}

Consider the map $\alpha : \mathfrak{R}_n^{\mathrm{uni}} \to \mathfrak{Q}_n$ defined by
\[
  (\rho_\Gamma)_{\Gamma} \mapsto \{\Part(\Gamma): \Gamma \text{ is a 1-layer tree and } \rho_{\Gamma} > 0\}.
\]
Given a collection of partitions $Q$, we obtain an assignment of radii to radially aligned curves by assigning the radius $\rho_r$ to $\Gamma$ if $r$ is the largest number such that $\Part(\rho_r) \in Q$. This gives a map $\beta : \mathfrak{Q}_n \to \mathfrak{R}_n^{\mathrm{uni}}$.  

\begin{proposition}
The maps $\alpha: \mathfrak{R}_n^{\mathrm{uni}} \to \mathfrak{Q}_n$ and $\beta : \mathfrak{Q}_n \to \mathfrak{R}_n^{\mathrm{uni}}$ are well-defined and are inverses.
\end{proposition}

\begin{proof}
To show $\alpha$ is well-defined, it suffices to show that its image is contained in $\mathfrak{Q}_n$. We do this by contradiction. Suppose that $Q = \alpha((\rho_{\Gamma})_{\Gamma})$ is not downward closed. We show that $(\rho_{\Gamma})_{\Gamma}$ is not universal. We can find some $P \in Q$ such that a minimal coarsening of $P$ is not in $Q$. Specifically, there will be a $P = \{p_1,...,p_k\} \in Q$ such that (up to reordering) $P' = \{p_1 \cup p_2,p_3,...,p_k\} \not\in Q$, as otherwise $Q$ will be downward closed. Let the 1-layer trees $\Gamma$ and $\Gamma'$ correspond to the partitions $P$ and the coarsened partition $P'$, respectively. Say $\Gamma$ has edge length $r$ and radius $r$ and $\Gamma'$ has edge length $s$ and radius 0. There is a 2-layer curve, say $\Tilde{\Gamma}$, that contracts to both $\Gamma$ and $\Gamma'$, and has edge lengths $s$ and $r$ and radius $s+r$ (see Figure \ref{fig:2_layer_construction}). If the radius was universal, then we see that by contracting the edge of length $r$, $\Gamma'$ must have a radius of $s$, not 0. Thus the radius is not universal, as claimed. 

\begin{figure}
    \centering
    \includegraphics[width=4in]{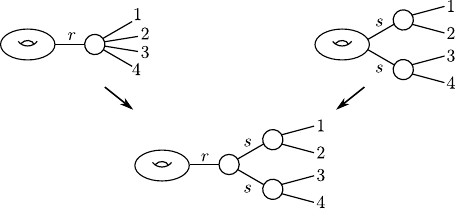}
    \caption{The tropical curves associated to the partition types $P$, $P'$, and $P' \prec P$ where $P = \{\{1,2\},\{3,4\}\}$, $P' = \{\{1,2,3,4\}\}$.}
    \label{fig:2_layer_construction}
\end{figure}

We now show that $\beta$ is well defined. First, note that given a basic radially aligned curve $\Gamma$, $\rho_{\Gamma}$ will be determined by the contractions to $\Gamma_{\delta_i}$ and $\Gamma_{e_i}$. To see this, note that the contraction $\Gamma_{\delta_i} \to \Gamma$ will send $\rho_j$ to 0 for all $j$, and the contraction $\Gamma_{e_i} \to \Gamma$ will send $\rho_j$ to $e_i$ if $j \geq i$ and 0 if $j < i$. As these maps arise from projections from a product, $\rho_{\Gamma}$ is uniquely determined by these contractions. Now pick $Q \in \mathfrak{Q}_n$. For any radially aligned $\Gamma$, Lemma \ref{lem:partitions_increase_strictly} and the fact that $\rho_{\Gamma}$ is determined by the contractions $\Gamma_{e_i} \to \Gamma$ imply that $\rho_{\Gamma}$ is actually a distance to the core. To see that this is universal, we need only check that single edge contractions are compatible, i.e., if $\Gamma$ has edge lengths $\{e_1,...,e_n\}$, then the contraction $\Gamma' \to \Gamma$ from sending $e_j$ to 0, where $\Gamma'$ has edge lengths $\{e_1,...,\hat{e_j},...,e_n\}$, is compatible. This compatibility follows immediately from contracting both $\Gamma$ and $\Gamma'$ to each of the $\Gamma_{e_i}$'s.

Finally, note that $\alpha$ is injective because the assignment of a radius to a radially aligned curve $\Gamma$ is uniquely determined by 1-layer trees. Furthermore, the discussion at the start of the previous paragraph shows that if $Q$ is the collection of partitions corresponding to the $1$-layer trees with non-zero radii in a universal radius, then $\rho_r$ is the radius determined by contractions to $1$-layer trees. This shows that $\beta = \alpha^{-1}$.
\end{proof}

\begin{theorem} \label{thm:contraction_to_Q}
For each $Q \in \mathfrak{Q}_n$, there is a diagram of stacks
\[
  \xymatrix{
     & \ol{\mathcal{M}}_{1,n}^{rad} \ar[dl] \ar[dr] & \\
    \ol{\mathcal{M}}_{1,n}  & & \ol{\mathcal{M}}_{1,n}(Q)
  }
\]
so that both arrows are proper and restrict to an isomorphism on $\mathcal{M}_{1,n}$.
\end{theorem}

\begin{proof}
Theorem \ref{thm:M_1n_rad_exists} gives us the arrow on the left; we only have to show that the arrow on the right has the claimed properties.

Let $(\rho_\Gamma) = \beta(Q)$ be the universal radius associated to $Q$. Note that for all $\Gamma$, $\rho_\Gamma$ is a radius of $\Gamma$ and $\Part(\rho_\Gamma) \in Q$, by construction.
By Lemma \ref{lem:characteristic_sections_as_tropical_lengths}, the $\rho_\Gamma$'s define a global section $\rho$ of the characteristic sheaf
of $\ol{\mathcal{M}}_{1,n}^{rad}$. Let $\pi : C \to \ol{\mathcal{M}}^{rad}_{1,n}$ be the universal curve. Theorem \ref{thm:contraction}
constructs a $Q$-stable family of curves $\ol{\pi} : \ol{C} \to \ol{\mathcal{M}}^{rad}_{1,n}$ associated to $\rho$, inducing the map $\ol{\mathcal{M}}^{rad}_{1,n} \to \ol{\mathcal{M}}_{1,n}(Q)$. This map restricts to an isomorphism on $\mathcal{M}_{1,n}$ as the maps $C \leftarrow \tilde{C} \rightarrow \ol{C}$ of Theorem \ref{thm:contraction} are isomorphisms where $\rho$ restricts to $0$.
\end{proof}

\section{Construction of the $Q$-stable moduli spaces}

\begin{theorem} \label{thm:construction}
For each $Q \in \mathfrak{Q}_n$, $\ol{\mathcal{M}}_{1,n}(Q)$ is a proper irreducible Deligne-Mumford stack over $\Z[1/6]$ containing $\mathcal{M}_{1,n}$.
\end{theorem}

Our argument is brief since we may reuse much of the proof for the analogous result for $m$-stable spaces in 
\cite[Theorem 3.8]{smyth_mstable}. In particular, to show that we have a Deligne-Mumford stack, it is enough to show that the moduli
functor $\ol{\mathcal{M}}_{1,n}(Q)$ is deformation open, bounded, and satisfies the valuative criterion for properness.
Boundedness is immediate from \cite[Lemma 3.9]{smyth_mstable}, since every $Q$-stable curve is $m$-stable for some $m < n$.
It is clear that $\ol{\mathcal{M}}_{1,n}(Q)$ contains $\mathcal{M}_{1,n}$.
The resulting stack is therefore irreducible, since every $Q$-stable curve is $m$-stable for some $m$, and all $m$-stable curves are limits of curves in $\mathcal{M}_{1,n}$.

\begin{theorem}[Deformation-openness] \label{thm:Q_deformation_open}
Let $S$ be a noetherian scheme and let $\pi: C \to S$ be a flat, projective morphism with one-dimensional fibers and let $\sigma_1, \ldots, \sigma_n$ be $n$ sections.
Then the set 
\[
  T = \{ s \in S \mid (\pi|_{\ol{s}} : C|_{\ol{s}} \to \ol{s}, \{ \sigma_i(\ol{s}) \}_{i = 1}^n) \text{ is $Q$-stable} \}
\]
is open.
\end{theorem}
\begin{proof}
As in \cite[Lemma 3.10]{smyth_mstable}, we may assume that the geometric fibers $C_{\ol{s}}$ of $\pi$ are reduced, connected, of arithmetic genus one, with only Gorenstein singularities, and that $H^0(C_s, \Omega^\vee_C(-\Sigma)) = 0$ since these are open conditions.

It remains to show that the locus in $S$ over which the level conditions hold is open. Since $S$ is Noetherian, we may establish openness by showing that this locus is constructible and stable under generization.
It is constructible since satisfaction of the level conditions is constant on combinatorial types (defined slightly ahead in Definition \ref{def:combinatorial_type}) and the curves with a given combinatorial type form a locally closed subset of $S$.

So assume $S$ is the spectrum of a DVR with closed point $0 \in S$ and generic point $\eta \in S$. We must show that if $(C_{\ol{0}}, \sigma_1(\ol{0}), \ldots, \sigma_n(\ol{0}))$ satisfies the level conditions, then so does $(C_{\ol{\eta}}, \sigma_1(\ol{\eta}), \ldots, \sigma_n(\ol{\eta}))$. Replacing $S$ by a finite base change if necessary, we may assume that the irreducible components of $C_{\ol{\eta}}$ are bijection with the irreducible components of $C$. 

Since level increases with the size of a subcurve, it is enough to check the subcurve level condition on minimal genus one subcurves. Let $E_{\ol{\eta}}$ be a minimal genus one subcurve of $C_{\ol{\eta}}$. Then the limit $Z$ of $E_{\ol{\eta}}$ in $C_{\ol{0}}$ contains
the minimal genus one subcurve $E_0$ of $C_{\ol{0}}$. Now $\lev(E_{\ol{0}}) \not\in Q$ by hypothesis, $\lev(E_{\ol{0}}) \preceq \lev(Z)$, and $Q$ is downward closed, so we have $\lev(Z) \not\in Q$. Finally, taking limits of the components of
$(C_{\ol{\eta}} - E_{\ol{\eta}}) \cup \Sigma$, we see that $\lev(E_{\ol{\eta}}) = \lev(Z)$, so $\lev(E_{\ol{\eta}}) \not\in Q$ as required.

The level condition on singularities holds trivially if $C_{\ol{\eta}}$ has only nodes, so suppose that $C_{\ol{\eta}}$ has an elliptic
$l$-fold singularity $q_{\ol{\eta}}$. Since elliptic $l_1$-fold singularities generize only to $l_2$-fold singularities with $l_2 \leq l_1$ or nodes, and nodes only generize to nodes and smooth points, the limit of $q_{\ol{\eta}}$ in $C_{\ol{0}}$ must be an elliptic $m$-fold singularity $q_{\ol{0}}$
with $m \geq l$. For each $i = 1, \ldots, l$ set $Z^i_{\ol{\eta}}$ to be the union of the $i$th rational branch $B^i_{\ol{\eta}}$ of $q_{\ol{\eta}}$ with all rational tails attached to this branch. Next, set $Z^i$ and $B^i$ to be the closures of $Z^i_{\ol{\eta}}$ and $B^i_{\ol{\eta}}$ in $C$ and write $Z^i_{lim}, B^i_{lim}$ for their restrictions to $C_{\ol{0}}$. Similarly, let $B^j_{\ol{0}}$ and $Z^j_{\ol{0}}$ be respectively the $j$th branch of $q_{\ol{0}}$ and the union of $B^j_{\ol{0}}$ with its rational tails for $j = 1,\ldots, m$. Observe that the partition of $\{1, \ldots, n \}$ induced by the markings in the $Z^i_{lim}$ agrees with the level of $q_{\ol{\eta}}$, and the partition of $\{1, \ldots, n \}$ induced by the $Z^j_{\ol{0}}$ agrees with the level of $q_{\ol{0}}$.

We claim that each $Z^j_{\ol{0}}$ factors through exactly one of the $Z^i_{lim}$. Note that the limit of at most one branch $B^i_{\ol{\eta}}$ of $q_{\ol{\eta}}$ can contain $B^j_{\ol{0}}$ since $B^j_{\ol{0}}$ is irreducible and each $B^i$ is an irreducible component of $C$. Observe that each irreducible component of $C_{\ol{0}}$ is connected via a nodal path to exactly one of the branches of $q_{\ol{0}}$. The remaining components of $Z^j_{\ol{0}}$ are connected to the $B^j_{\ol{0}}$ via a nodal path, so we conclude that these belong to $Z^i_{lim}$ as well. The claim follows.

Therefore the level of $q_{\ol{\eta}}$ is a coarsening of the level of $q_{\ol{0}}$, that is, $\lev(q_{\ol{\eta}}) \prec \lev(q_{\ol{0}})$. Since $\lev(q_{\ol{0}}) \in Q$ and $Q$ is downward closed, $q_{\ol{\eta}} \in Q$ too, and we are done.
\end{proof}

\begin{theorem}
The stack $\ol{\mathcal{M}}_{1,n}(Q)$ is universally closed.
\end{theorem}
\begin{proof}
Since $\mathcal{M}_{1,n}$ is dense in $\ol{\mathcal{M}}_{1,n}(Q)$, it is enough to show that limits of families of smooth $n$-pointed curves
admit $Q$-stable limits. Suppose $S$ is the spectrum of a discrete valuation ring with generic point $\eta$ and $\pi_\eta : C_\eta \to \eta$ is a
smooth and proper family of $n$-pointed curves. Since $\ol{\mathcal{M}}_{1,n}^{rad}$ is proper, after replacing $S$ by a finite cover if necessary,
we may find a limit basic radially aligned log curve $\pi^{rad} : C^{rad} \to S$ extending $C_\eta \to \eta$. The image of $C^{rad} \to S$ under
the map of Theorem \ref{thm:contraction_to_Q} corresponding to $Q$ gives the required limit $Q$-stable curve.
\end{proof}

We now recall the definition of balanced subcurve \cite[Definition 2.11]{smyth_mstable}.

\begin{definition} \label{def:nodal_distance}
Given a connected nodal curve $E$ and components $F_1$, $F_2$, we say that the \emph{nodal distance} $l(F_1,F_2)$ \emph{from $F_1$ to $F_2$} is the least number of edges in a path from $F_1$ to $F_2$ in the dual graph of $E$. If $p \in E$ is a smooth point, then there is a unique irreducible
component $F$ of $E$ containing $p$, and we write $l(-,p)$
instead of $l(-,F)$.
\end{definition}

\begin{definition} \label{def:balanced}
If $(E, \{p_i\}_{i = 1}^m)$ is a semistable curve of arithmetic genus one, we say that $(E, \{p_i\}_{i=1}^m)$
is \emph{balanced} if
\[
  l(Z,p_1) = l(Z, p_2) = \cdots = l(Z,p_m),
\]
where $Z \subseteq E$ is the minimal elliptic subcurve of $E$.
\end{definition}

We remark that the nodal distance is an integer, while the distances $\lambda(v)$ defined for genus one log curves take values in the characteristic monoid of the base. The two notions are related by the fact that if
\begin{enumerate}
    \item $A$ is the spectrum of a DVR with uniformizer $t$,
    \item $S = \Spec A$ is given the log structure associated to the chart $\N \delta \to A$ sending $\delta \mapsto t$,
    \item $\pi : C \to S$ is a log curve of genus one with smooth generic fiber
    \item $C$ has regular total space
    \item $Z \subseteq C$ is the minimal genus one subcurve
\end{enumerate}
then all smoothing parameters of the nodes of $C_0$ are equal to $\delta$ and $\lambda(F) = l(Z, F)\delta$ for all irreducible components $F$ of $C$.

\begin{theorem}
The stack $\ol{\mathcal{M}}_{1,n}(Q)$ is separated.
\end{theorem}
\begin{proof}
We must show that given a pair of $Q$-stable families $\pi : C \to S$ and $\pi' : C' \to S$ over the spectrum of a discrete valuation
ring with generic point $\eta$ and special point $x$ that an isomorphism $\psi : C|_\eta \to C'|_\eta$ of pointed curves extends
to all of $S$. As in \cite[3.3.2]{smyth_mstable}, we may assume that there is a flat
and proper pointed semistable nodal curve $C^{ss} \to S$ with regular total space and a diagram of pointed $S$-schemes
\[
  C \overset{\phi}{\longleftarrow} C^{ss} \overset{\phi'}{\longrightarrow} C'
\]
where $\phi$ and $\phi'$ are proper birational morphisms, and it will be enough to show that the exceptional
loci of $\phi$ and $\phi'$ coincide.

If $C|_{\ol{x}}$ or $C'|_{\ol{x}}$ is nodal, then we may argue exactly as in \cite[3.3.2]{smyth_mstable} to conclude. Therefore we may assume that $C|_{\ol{x}}$ and $C'|_{\ol{x}}$ possess elliptic Gorenstein singularities
$p$ and $p'$ respectively. Let $E = \phi^{-1}(p)$ and $E' = \phi^{-1}(p')$. As in \cite{smyth_mstable},
we know that $E$ and $E'$ are balanced, with $E$ (resp. $E'$) consisting of all components of $C^{ss}|_{\ol{x}}$
with nodal distance to the core of $C^{ss}|_{\ol{x}}$ less than $l$ (resp. less than $l'$).
Without loss of generality we may assume $E \subseteq E'$. If the containment is proper, then $\phi(E') \subseteq C|_{\ol{x}}$
is a subcurve of genus one containing $p$ and all of its branches. Examining the dual graphs of the various curves over $\ol{x}$, we
have $\lev(p') = \lev(\phi(E'))$. Then since $C$ is $Q$-stable, $\lev(\phi(E')) \not\in Q$. On the other hand $\lev(p') \in Q$, since $C'$ is $Q$-stable. This is a contradiction, so we have $E = E'$.

The remainder of the argument follows as in \cite[3.3.2]{smyth_mstable}.
\end{proof}

\section{Classification of semistable Gorenstein modular compactifications of $\mathcal{M}_{1,n}$}
\label{sec:classification}

Our goal in this section is to prove Theorem \ref{thm:classification}, classifying the modular compactifications of $\mathcal{M}_{1,n}$. To aid our classification, we introduce the notion of the combinatorial type of a curve in $\mathcal{U}_{1,n}$. This is analogous
to the dual graph of a nodal curve, with the difference that the combinatorial type also keeps track of elliptic $m$-fold singularities.

\begin{definition} \label{def:combinatorial_type}
Let $C$ be a connected, proper, reduced, 1-dimensional scheme over an algebraically closed field $k$ with (at worst) nodes and elliptic Gorenstein singularities. The \emph{combinatorial type of $C$} consists of the following data:
\begin{enumerate}
 \item A set $V$ of \emph{vertices}, equal to the set of components of $C$;
 \item A set $E$ of \emph{singularities}, equal to the set of singular points of $C$;
 \item A \emph{genus function} $g : V \cup E \to \N$ taking a component of $C$ to its genus of its normalization and taking each singularity of $C$ to its genus as a singularity.
 \item An \emph{incidence function} $i : V \times E \to \{ 0, 1 \}$ taking $(v,e) \mapsto 1$ if $e \in v$ and $0$ otherwise.
 \item A \emph{marking function} $x : V \to 2^{\{ 1,\ldots, n\}}$ taking a component $v$ to the set of indices of the markings incident to $v$.
\end{enumerate}

Two combinatorial types $\Gamma_1 = (V_1, E_1, g_1, i_1, x_1)$ and $\Gamma_2 = (V_2, E_2, g_2, i_2, x_2)$ are \emph{isomorphic} if there is a bijection $f : V_1 \cup E_1 \to V_2 \cup E_2$ so that
\begin{enumerate}
  \item $f(V_1) \subseteq V_2$ and $f(E_1) \subseteq E_2$;
  \item $x_1 = x_2 \circ f$;
  \item $g_1 = g_2 \circ f$;
  \item $i_1(v,e) = i_2(f(v), f(e))$ for all $(v,e) \in V_1 \times E_1$.
\end{enumerate}

\end{definition}

Let $\mathcal{Z}_\Gamma$ be the locus in $\mathcal{U}_{1,n}$ of curves with combinatorial type $\Gamma$. These loci have a natural structure of locally closed substack of $\mathcal{U}_{1,n}$. (This follows from the fact that the deformations of a curve preserving its singularities form a closed subspace of the full deformation space of the curve. See \cite[Lemma 2.1]{smyth_mstable2}, for example.) Note that for each $n$, there is a finite set of isomorphism classes of combinatorial types of curves in $\mathcal{U}_{1,n}$ and altogether the $\mathcal{Z}_\Gamma$'s form a stratification of $\mathcal{U}_{1,n}$ into locally closed substacks.

Let $\mathcal{M}$ be a modular compactification in our sense. Since $\mathcal{M}$ is assumed to be an open substack of $\mathcal{U}_{1,n}$, it is uniquely determined by its points. Our strategy is to show that $\mathcal{M}$ must be a union of the
$\mathcal{Z}_\Gamma$'s. Then, analyzing the possible limits of curves, we will find that the choices of combinatorial types making up $\mathcal{M}$ necessarily agree with a $Q$-stability condition.

\begin{lemma}
$\mathcal{U}_{1,n}$ is the union of the $\ol{\mathcal{M}}_{1,n}(m)$s.
\end{lemma}
\begin{proof}
Suppose that $(\pi : C \to S, \sigma_1, \ldots, \sigma_n)$ is a family in $\mathcal{U}_{1,n}$. We want to show that $S$ possesses an open cover so that the restriction of $C$ to each part the cover factors through some $\ol{\mathcal{M}}_{1,n}(m)$. Let $\ol{s}$ be a geometric point of $S$. Then the fiber $C_{\ol{s}}$ of $\pi$ over $\ol{s}$ is a Gorenstein curve of genus one with $n$ distinct marked points and
no infinitesimal automorphisms. Recall that the only Gorenstein singularities of genus less than or equal to one are the elliptic Gorenstein singularities and the node. If $C_{\ol{s}}$ has an elliptic $m$-fold singularity for some $m$, then since $C$ has no infinitesimal automorphisms, the number of markings and nodes on the minimal genus one subcurve must be at least $m + 1$. It follows that $C_{\ol{s}} \in \ol{\mathcal{M}}_{1,n}(m)$. If $C_{\ol{s}}$ does not have an elliptic Gorenstein singularity, then $C_{\ol{s}} \in \ol{\mathcal{M}}_{1,n} = \ol{\mathcal{M}}_{1,n}(0)$. Since the
$\ol{\mathcal{M}}_{1,n}(m)$'s are deformation open \cite{smyth_mstable}, for each $\ol{s}$ there is an open neighborhood $U_{s}$ of the image of $\ol{s}$ in $S$ so that $C|_{U_s}$ factors through one of the stacks $\ol{\mathcal{M}}_{1,n}(m)$.
\end{proof}

\begin{lemma}
Let $\Gamma$ be a combinatorial type. Then $\mathcal{Z}_\Gamma$ is irreducible.
\end{lemma}
\begin{proof}
If $\Gamma$ possesses no $m$-fold points, then $\mathcal{Z}_\Gamma$ is a product of copies of $\mathcal{M}_{g,n}$'s coming from the vertices of $\Gamma$. Since each of the stacks $\mathcal{M}_{g,n}$ is geometrically irreducible, so is $\mathcal{Z}_\Gamma$.

Next, let $\Gamma$ be a combinatorial type consisting of a single elliptic $m$-fold point with $k$ rational branches $E_1, \ldots, E_k$ with $n_1, \ldots, n_k$ markings, respectively. Let $n = n_1 + \cdots + n_k$.

Let
\[
  A = \left\{ (f_i(t_i))_{i = 1}^m \in \prod_{i = 1}^m \Z\left[1/6,t_i\right] \,\, \middle| \,\, f_i(0) = f_j(0) \text{ for all }i, \text{ and } \sum_{i = 1}^m f_i'(0) = 0 \right\}
\]
This gives a standard affine model of the $m$-fold point with rational branches.
Form a proper curve $D \to \Spec \Z[1/6]$ by gluing $\Spec \Z[1/6, t_i^{-1}]$ to the $i$th branch of $\Spec A$ for each $i$. If $C$ is a minimal unmarked Gorenstein curve of genus one over an algebraically closed field with an $m$-fold point, then $C$ appears as a geometric fiber of $D$.

Now, let
\[
  S = \prod_{i = 1}^m \left( \Spec \Z[1/6, s_{i,1}, \ldots, s_{i,n_i}] - \Delta_{n_i} \right)
\]
where $\Delta_{n_i}$ is the locus where any pair of coordinates of $\mathbb{A}^{n_i}_{\Z[1/6]} = \Spec \Z[1/6, s_{i,1}, \ldots, s_{i,n_i}]$ coincide. We construct a family of pointed curves $D \times S \to S$ by taking the $j$th marking on the $i$th branch of $D$ to be located at
$t_i^{-1} = s_{i, j}$.

Now, every pointed curve of type $\Gamma$ appears as some geometric fiber of the family $D \times S \to S$. Therefore, the image of $S$ in
$\mathcal{U}_{1,n}$ under the map induced by $D \times S \to S$ is precisely $\mathcal{Z}_\Gamma$. Since $S$ is irreducible, the result follows.

Finally, consider a general $\Gamma$. Let $\Gamma_{min}$ be the combinatorial type of the minimal genus one sub-combinatorial type of $\Gamma$ with markings at the outgoing edges. Clearly, $\mathcal{Z}_{\Gamma}$ is a product of $\mathcal{M}_{0,n}$'s and $\mathcal{Z}_{\Gamma_{min}}$, all of which are already known to be geometrically irreducible, so $\mathcal{Z}_\Gamma$ is irreducible too.
\end{proof}

The following lemma is the crucial one: it reduces the classification of Gorenstein compactifications to combinatorics.

\begin{lemma} \label{lem:union_of_loci}
$\mathcal{M}$ is a union of the $\mathcal{Z}_\Gamma$'s.
\end{lemma}
\begin{proof}
It suffices to show that if $\mathcal{M}$ shares a geometric point with $\mathcal{Z}_\Gamma$ for some $\Gamma$, then $\mathcal{M}$ contains all points of $\mathcal{Z}_\Gamma$. By the previous lemma, for any pair of geometric points
$C_p \in \mathcal{Z}_\Gamma(\Spec k(p)), C_q \in \mathcal{Z}_\Gamma(\Spec k(q))$, there are families of curves
$C_S\in \mathcal{Z}_\Gamma(S), C_T \in \mathcal{Z}_\Gamma(T)$ where $S$ and $T$ are spectra of discrete valuation rings so that
\begin{enumerate}
  \item $S$ and $T$ have isomorphic geometric generic points,
  \item $C_S$ is isomorphic to $C_T$ over this common geometric generic point,
  \item There is a map $\Spec k(p) \to S$ onto the special point of $S$ along which $C_S$ pulls back to $C_p$,
  \item There is a map $\Spec k(q) \to T$ onto the special point of $T$ along which $C_T$ pulls back to $C_q$.
\end{enumerate}

We know $\mathcal{M}$ is closed under generization and has a specialization for any 1-dimensional family. It follows that if $\mathcal{M}$ contains $C_p$, then $\mathcal{M}(S)$ contains $C_S$. To conclude, we show that $\mathcal{M}(T)$
contains $C_T$ too.

Let $\eta$ be the generic point of $T$, and, replacing $T$ by a finite base change if necessary, let $C_T' \in \mathcal{M}(T)$ be the unique limit of $C_T|_\eta$ in $\mathcal{M}$. We have an isomorphism of $\eta$-schemes $C _T'|_{\eta} \cong C_T|_{\eta}$.
Since $\mathcal{U}_{1,n}$ is a union of the open substacks $\ol{\mathcal{M}}_{1,n}(m)$, there is some $m$ so that $C_T'$ lives in $\ol{\mathcal{M}}_{1,n}(m)(T)$. Considering $C_T'|_{\ol{\eta}}$, we conclude that curves of combinatorial type $\Gamma$ are $m$-stable. In particular, both $C_T$ and $C_T'$ are families in $\ol{\mathcal{M}}_{1,n}(m)$.
Since $\ol{\mathcal{M}}_{1,n}(m)$ is separated, we conclude that $C_T' \cong C_T$ over $T$, completing the proof.
\end{proof}

Our strategy now is to produce families of curves witnessing enough of the relationships between the loci $\mathcal{Z}_\Gamma$
that $\mathcal{M}$ is forced to be $Q$-stable.

\begin{definition}
Let
\[
  \mathcal{P} : P_1 \prec P_2 \prec \cdots \prec P_k
\]
be a strictly increasing chain of partitions of $\{1, \ldots, n\}$ not including the partition $\{ \{ 1 \}, \cdots, \{ n \} \}$. We say that a family of radially aligned curves $\pi : C \to S$ is a \emph{test curve of type $\mathcal{P}$ centered at a geometric point $s$ of $S$} if
\begin{enumerate}
  \item $(\pi : C \to S, s)$ satisfies the conclusions of Theorem \ref{thm:uniform_charts};
  \item The tropicalization of the central fiber $\trop(C|_s)$ has a basic radially aligned log structure;
  \item The log structure on $S$ is divisorial, that is, it is the log structure associated to a normal crossings divisor \cite[(1.5)]{kato_log_structures};
  \item The tropicalization of the central fiber has partition type $\mathcal{P}$ (Definition \ref{def:partition_type}). 
\end{enumerate}
\end{definition}

\begin{lemma} \label{lem:test_curves_exist}
For any strictly increasing chain of partitions of $\{1, \ldots, n\}$
\[
  \mathcal{P} : P_1 \prec P_2 \prec \cdots \prec P_k
\]
not including the partition $\{ \{ 1 \}, \cdots, \{ n \} \}$, there is a test curve of type $\mathcal{P}$.
\end{lemma}

\begin{proof}
Choose an algebraically closed field $\kappa$.
Pick a smooth genus one curve $E$ over $\kappa$ arbitrarily. Add a rational component $Z^{(1)}_S$ for each subset $S$ of $P_1$ and attach them nodally to $E$ at distinct smooth points. Repeat this process for each $i = 2, \ldots, k$, adding
components $Z^{(i)}_S$ for each subset $S \in P_i$, where $Z^{(i)}_S$ is nodally attached to the unique component $Z^{(i-1)}_{T}$ where $T \in P_{i - 1}$ is the subset containing $S$. Finally, mark points $p_1, \ldots, p_n$ where
each $p_i$ is a smooth point of the component $Z^{(k)}_S$ where $S$ is the subset of $P_k$ containing $i$. Call the whole pointed nodal curve we have constructed $C_0$. Give $S_0 = \Spec \kappa$ the log structure associated to the map $\bigoplus_{i = 1}^k \N e_i \to \kappa$ sending everything to 0 except for the identity element. Choose a log structure on $\pi_0 : C_0 \to S_0$ compatible with the log structure on $S$ making $\pi_0 : C_0 \to S_0$ into a log curve so that
\begin{enumerate}
  \item the edges between the $Z^{(1)}_S$'s and $E$ are all of length $e_1$;
  \item for $i = 2, \ldots, k$ the edges between the $Z^{(i)}_S$'s and $Z^{(i - 1)}_S$'s are all of length $e_i$.
\end{enumerate}

Stabilize to obtain a log curve $\pi_0^s : C^s_0 \to S_0$. Since $\pi_0^s$ is radially aligned with the basic radially aligned log structure, there is an associated map $f : S_0 \to \ol{\mathcal{M}}_{1,n}^{rad}$. Choose a factorization of it through an \'etale chart $U \to \ol{\mathcal{M}}_{1,n}^{rad}$. Now, set $S$ to be a neighborhood of the image $s$ of $S_0$ in $U$ using Theorem \ref{thm:uniform_charts}.

We now claim that the pullback of the universal curve of $\ol{\mathcal{M}}_{1,n}^{rad}$ to $S$ gives the desired test curve. Properties (i) and (ii) hold by construction. Property (iii) holds since $S$ was chosen to be an \'etale neighborhood of a point of the chart
$U$, which has a divisorial log structure.

It remains to check property (iv). If $i$ is an integer with $1 \leq i < k$, there is at least one $Z^{(i)}_S$ with at least two descendants, since $P_i \neq P_{i+1}$. If $i = k$, all of the $Z^{(i)}_S$'s possess at least three marked points. Either way, for each integer $i = 1, \ldots, k$ there is a component $Z^{(i)}_{S}$ that survives the stabilization step, and it lives at radius $\lambda(Z^{(i)}_S) = e_1 + \cdots + e_i$ by construction. Conversely, every radius is of this form. Now, if we subdivide $\trop(C|_s)$ where $\lambda = e_1 + \cdots + e_i$, the effect is to reintroduce the components $Z^{(i)}_{S}$ for $S \in P_i$. Deleting the locus where $\lambda < e_1 + \cdots + e_i$ leaves us with the trees rooted at the $Z^{(i)}_S$'s as $S$ varies through the elements of $P_i$. By construction, for each $S \in P_i$ the tree rooted at $Z^{(i)}_S$ contains the markings indexed by $S$. Therefore, the partition type of $\trop(C|_s)$ is precisely $\mathcal{P}$.
\end{proof}

With test curves in hand, we may form several families of contracted curves. Exactly one of these will be $Q$-stable for any $Q$.

\begin{lemma} \label{lem:q_stable_test_curve_contractions}
Let $\pi : C \to S$ be a test curve of type $\mathcal{P}$ centered at $s$. Let $\rho_0 = 0 < \rho_1 < \cdots < \rho_k$ be the distinct radii of the tropicalization of the central fiber. For each $i = 0, \ldots, k$, let $\ol{C}_{i} \to S$ be the contraction of $C \to S$ associated to the steep mesa with radius $\rho_i$ and let $\Gamma_i$ be the combinatorial
type of the fiber of $\ol{C}_i$ over $s$.

Let $Q \in \mathfrak{Q}_{1,n}$ be a downward closed set of partitions. Then there is exactly one index $i$ so that $\Gamma_i$ is $Q$-stable, namely
the greatest index $i$ for which $P_i \in Q$. 
\end{lemma}
\begin{proof}
Observe that
\begin{enumerate}
  \item $\ol{C}_0|_s$ is nodal and its minimal elliptic subcurve has level $P_1$;
  \item For $0 < i < k$, $\ol{C}_i|_s$ possesses an elliptic singularity of level $P_i$ and its minimal elliptic subcurve has level $P_{i+1}$;
  \item For $i = k$, $\ol{C}_i|_s$ possesses an elliptic singularity of level $P_k$ and its minimal elliptic subcurve (namely $\ol{C}_k|_s$) has level $\{ \{1\}, \ldots, \{ n\} \}$.
\end{enumerate}
The result follows immediately.
\end{proof}

With some more work, we see that exactly one of the contracted curves belongs to our arbitrary modular compactification $\mathcal{M}$ as well.

\begin{lemma} \label{lem:exactly_one_contraction}
Choose notation as in Lemma \ref{lem:q_stable_test_curve_contractions}. Then $\mathcal{M}$ contains exactly one of the families $\ol{C}_i \to S$ and exactly one of the loci $\mathcal{Z}_{\Gamma_i}$.
\end{lemma}
\begin{proof}
Choose $t$ to be a point of $S$ generizing $s$ over which $C$ is smooth. Let $T \to S$ be a map from the spectrum of a DVR with special point $x$ mapping to $s$ and generic point $\eta$ mapping to $t$.

Replacing $T$ by a finite base change if necessary, find the limit curve $C_{\mathcal{M}} \to T$ in $\mathcal{M}$ of the smooth curve $C|_\eta$. After a second base change if necessary, pick a regular family of semistable curves $C^{ss} \to T$ dominating $C_{\mathcal{M}} \to T$ and each of the families $\ol{C}_i|_T \to T$, formed by subdividing and contracting the $i$th radius of $C$. Note that $\ol{C}_0 = C|_T$. Let $E$ be the exceptional locus of $\phi : C^{ss} \to C_{\mathcal{M}}$ and let $E_i$ be the exceptional locus of $\phi_i : C^{ss} \to \ol{C}_i|_T$ for each $i$.

If $C_{\mathcal{M}}$ is stable then both $C|_T$ and $C_{\mathcal{M}}$ are stable limits of $C|_{\eta}$, so they must agree.

Otherwise, $C_{\mathcal{M}}$ possesses a unique elliptic $m$-fold point $p$. By \cite[Proposition 2.12]{smyth_mstable}, $\phi^{-1}(p)$ is a balanced subcurve
of $C^{ss}|_x$, with $\phi^{-1}(p)$ consisting of all components of $C^{ss}|_x$ whose nodal distance from the core of $C^{ss}|_x$ is less than some integer $l$. Since $C_{\mathcal{M}}$ possesses no infinitesimal automorphisms, at least one of the components of $C^{ss}|_x$ of distance exactly $l$ from the core of $C^{ss}|_x$ has at least three special points. Then $E$ is the union
of $\phi^{-1}(p)$ with the semistable components of $C^{ss}|_x$ disjoint from $\phi^{-1}(p)$.

We may repeat this argument for each of the exceptional loci $E_i$, $i > 0$. Therefore, for each $i$,
\begin{enumerate}
  \item there is a balanced subcurve $F_i$ of $C^{ss}|_x$ consisting of all components of nodal distance from the core of $C^{ss}|_x$ less than some integer $l_i$,
  \item there is a component of $C^{ss}|_x$ of distance exactly $l_i$ from the core with at least three markings, and
  \item $E_i$ is the union of $F_i$ with the semistable components of $C^{ss}|_x$ disjoint from $F_i$.
\end{enumerate}

Moreover, we claim that the loci $E_i$ exhaust the subcurves of $C^{ss}|_x$ with these properties. To see this, suppose $F'$ is a balanced subcurve of $C^{ss}|_x$
whose components are those with nodal distance less than $l'$, and so that there is a stable component $G$ of $C^{ss}|_x$ of distance $l'$ from the core. Then $\phi(G)$ maps to a stable component of $C|_x$.
The component $G$ lives at the distance $\rho_i$ from the core of $C|_x$ for some $i$. Then for this same $i$, we have that $l' = l_i$ and the rest follows.

Therefore $E = E_i$ for some $i > 0$. Since both $C_{\mathcal{M}}$ and $\ol{C}_i|_T$ are normal and obtained from contracting the same locus of $C^{ss}$, they are isomorphic.

This exhibits a curve, namely $\ol{C}_i|_x$, in $\mathcal{Z}_{\Gamma_i}$ in $\mathcal{M}$. By Lemma \ref{lem:union_of_loci}, $\mathcal{M}$ contains the whole locus. By separatedness, $\mathcal{M}$ cannot contain any of the other $\ol{C}_j|_T$ for $j \neq i$. Therefore $\mathcal{M}$ cannot contain $\mathcal{Z}_{\Gamma_{j}}$ for any $j \neq i$.

Finally, we wish to show that in fact $\mathcal{M}$ contains the entire family $\ol{C}_i \to S$. This will follow from the closure
of $\mathcal{M}$ under generization. Write $\bigoplus_i e_i \oplus \bigoplus_j \delta_j$ for $\ol{M}_{S,s} \cong \Gamma(S, \ol{M}_S)$ so that the $e_i$'s are the differences between consecutive radii, as in Definition \ref{def:basic_radially_aligned}. The base $S$ of $C \to S$ possesses a stratification by locally closed subsets $\{ W_I \}$ indexed by subsets
$I \subseteq \{ 1, \ldots, k \}$, where
\[
  W_I = \bigcap_{j \in I} \mathrm{Supp}(e_j) \cap \bigcap_{j \in I^c} (S - \mathrm{Supp}(e_j)).
\]
Since $C \to S$ satisfies the conclusions of Theorem \ref{thm:uniform_charts}, the tropicalization of the fibers of $C$ is constant on the
subsets $W_I$. It follows that the same holds of the combinatorial types of the fibers of $\ol{C}_i \to S$. Since the log structure of $S$
is divisorial, for each $I \subseteq \{1, \ldots, k \}$, there is a generization $\eta_{I} \to s$ where $\eta_I \in W_I$. Since $\mathcal{M}$
is closed under generization and $\mathcal{M}$ contains $\ol{C}_i|_s$, $\mathcal{M}$ must also contain $\ol{C}_i|_{\eta_I}$. Then, since
membership in $\mathcal{M}$ is determined by combinatorial type, $\mathcal{M}$ contains all of $\ol{C}_i|_{W_I}$. We conclude that the whole
family $\ol{C}_i \to S$ belongs to $\mathcal{M}$.
\end{proof}

\begin{remark}
At first glance, the conclusion of the lemma looks different for test curves whose chains of partitions agree but whose cores have different combinatorial types, since their combinatorial types $\Gamma_0$ will be distinct. Such pairs of test curves share combinatorial types $\Gamma_i$ for $i \geq 1$, so, comparing the conclusion of the lemma for the various test curves, either
$\mathcal{M}$ chooses to contract the core of all test curves of type $\mathcal{P}$ or it contracts the core of none of them.
\end{remark}

In the case that the chain $\mathcal{P}$ consists of a single partition $P$, Lemma \ref{lem:exactly_one_contraction} tells us 
there are only two ``choices" that $\mathcal{M}$ could make: either $\mathcal{M}$ contains the contracted curve $\ol{C}_1$ or $\mathcal{M}$ contains the uncontracted curve $C$. In the former case,
say that $\mathcal{M}$ \emph{contracts} $P$. Our next claim is that for any test curve
$C$, the combinatorial type of the contraction of $C$ included by $\mathcal{M}$ is determined solely by the partitions which $\mathcal{M}$
contracts.

\begin{lemma} \label{lem:determined_by_length_1_chains}
Choose notation as in Lemma \ref{lem:exactly_one_contraction}. Let $i$ be the index of the family $\ol{C}_i \to S$ contained in $\mathcal{M}$.
Let $j$ be any integer with $1 \leq j \leq k$. Then $j \leq i$ if and only if $\mathcal{M}$ contracts $P_j$.
\end{lemma}
\begin{proof}
Write $\bigoplus_i e_i \oplus \bigoplus_j \delta_j$ for $\ol{M}_{S,s} \cong \Gamma(S, \ol{M}_S)$ so that the $e_i$'s are the differences between consecutive radii of $\trop(C)$, as in Definition \ref{def:basic_radially_aligned}.
Let $C' \to S'$ be the restriction of $C \to S$ to the complement of the support of $e_1, \ldots, \hat{e}_j, \ldots, e_k$. Choose
$s'$ to be a point of the support of $e_j$ inside $S'$. It is not difficult to check that the resulting family is a test family of type $P_j$ (i.e.,
the chain of partitions of length one containing just the partition $P_j$) whose log structure is generated by $e_j$. The unique non-zero
radius of this test family is $e_j$.
Since the contraction of a family of mesa curves commutes with base change, $\ol{C}_i|_{S'}$ will be the contraction
of $C'$ with respect to $\rho|_{S'}$. Now, the restriction of $\rho_i$ from $S$ to $S'$ is either $e_j$ if $j \leq i$ or 0 if $i < j$. In the first case, $\mathcal{M}$ contracts $P_j$. In the latter $\mathcal{M}$ does not.
\end{proof}

\begin{lemma} \label{lem:everything_is_test_curve}
Every combinatorial type $\Gamma$ in $\mathcal{U}_{1,n}$ is the combinatorial type of some contraction of a test curve.
\end{lemma}
\begin{proof}
Choose a representative curve $C_0$ for $\Gamma$. Find a smoothing $C \to T$ of $C_0$ over the spectrum of a discrete valuation ring with special point $x$ and generic point $\eta$. Replacing $T$ by a finite base change if necessary, let $C^{rad} \to T$ be the limit radially aligned curve of $C|_{\eta} \to \eta$ in $\ol{\mathcal{M}}_{1,n}^{rad}$ with the basic radially aligned log structure.

Arguing as in Lemma \ref{lem:test_curves_exist}, we can
find a family $C^{test} \to S$ centered at $s$ with central fiber $C^{rad}|_x$ by choosing a sufficiently small \'etale neighborhood
of $C^{rad}|_x$ in an \'etale chart for $\ol{\mathcal{M}}_{1,n}^{rad}$. This neighborhood may be assumed to satisfy the conditions of Theorem \ref{thm:uniform_charts} and its log
structure will be divisorial since it is pulled back from the log structure of $\ol{\mathcal{M}}_{1,n}^{rad}$ along an \'etale morphism.
This family $C^{test} \to S$ is therefore a test curve of some type, with the type given by the partitions induced by the various radii
of the tropicalization of its central fiber.

Moreover, replacing $T$ by a finite base change if necessary, we may assume that $C^{rad} \to T$ is pulled back from $C^{test} \to S$.
We are now in exactly the situation of the proof of Lemma \ref{lem:exactly_one_contraction}, so we conclude that
$C_0$ is one of the contractions of the central fiber of $C^{test}$, as required.
\end{proof}

Now that everything is in place, our classification result follows easily.

\begin{proof}[Proof of Theorem \ref{thm:classification}]
Let $Q_{\mathcal{M}}$ be the set of non-discrete partitions $P$ of $\{1, \ldots, n \}$ so that $\mathcal{M}$ contracts $P$. By Lemma \ref{lem:determined_by_length_1_chains},
$Q_{\mathcal{M}}$ is downward closed. 

By Lemma \ref{lem:determined_by_length_1_chains} and Lemma \ref{lem:q_stable_test_curve_contractions} the combinatorial types that appear as contractions of test curves contained in $\mathcal{M}$ and $\overline{\mathcal{M}}_{1,n}(Q_\mathcal{M})$ are the same. On the other hand, by Lemma \ref{lem:everything_is_test_curve}, every combinatorial type appears as $\Gamma_i$ for some test curve $C^{test}$. It follows that $\mathcal{M}$ contains precisely the $Q_{\mathcal{M}}$-stable curves.
\end{proof}

\section{Interpolation of $Q$-stable spaces} \label{sec:interpolation}

We learned in Section \ref{sec:radially_aligned_contractions} that the choices of universal radii are in bijection with compatible choices of radii on 1-layer
trees. Each 1-layer tree $\Gamma$ has two possible radii: the zero radius and a unique nonzero radius; let's call it $e_\Gamma$. Thinking tropically, it is tempting to choose
radii intermediate to $0$ and $e_\Gamma$ to contract. If we do this and follow the image of the universal curve under the resulting contraction, we find Artin stacks that interpolate between the $Q$-stable spaces. The moduli problems of these stacks admit an elegant description, so we will define them and verify their algebraicity directly prior to a tropical construction.

We therefore imagine choosing an intermediate radius for each 1-layer tree $\Gamma$ in the form $c_\Gamma e_\Gamma$ with $c_\Gamma \in [0,1]$. The possible choices of
such radii form a cube complex, which we now describe directly.

\begin{definition}
Let $X_{1,n}$ be the subset of $[0,1]^{\PPart(n)}$ of tuples $(c_P)_{P \in \PPart(n)}$ with the properties:
\begin{enumerate}
  \item whenever $P_1, P_2$ are two partitions of $\{1, \ldots, n\}$ with $P_1 \prec P_2$ and $c_{P_2} > 0$, then $c_{P_1} = 1$;
  \item $c_{\{ \{1\}, \ldots, \{n\} \}} = 0$.
\end{enumerate}

We give $X_{1,n}$ the structure of a cube complex whose open cells are the intersections of the open cells of $[0,1]^{\PPart(n)}$ with $X_{1,n}$.
\end{definition}

Note that the vertices of $X_{1,n}$ are in bijection with $\mathfrak{Q}_n$: the correspondence is given by taking a subset $Q$ of $\PPart(n)$ to its indicator function.

\begin{example}
$X_{1,3}$ consists of the tuples
\[
  (x, y, z, w) = (c_{\{\{1,2,3\}\}}, c_{\{\{1\}, \{2,3\}\}}, c_{\{\{2\}, \{1,3\}\}}, c_{\{\{3\}, \{1,2\}\}})
\]
in $[0,1]^4$ where $y = z = w = 0$ unless $x = 1$. Then $X_{1,3}$ consists of the 1-dimensional cube $\{ (x, 0, 0, 0) \mid x \in [0,1] \}$ attached to
the 3-dimensional solid cube $\{ (1,y,z,w) \mid y,z,w \in [0,1] \}$. The latter cube fills in the interior of the cube visible in Figure \ref{fig:q_stable_lattice_3}.
\end{example}

The stacks associated to these imagined universal radii are as follows.

\begin{definition}
Let $(c_P)_{P \in \PPart(n)}$ be an element of $X_{1,n}$. Let
\begin{align*}
  Q_{sing} &= \{ P \in \PPart(n) \mid c_P > 0 \} \\
  Q_{curve} &= \{ P \in \PPart(n) \mid c_P < 1 \}.
\end{align*}
An $n$-pointed family of Gorenstein curves $(\pi : C \to S , \sigma_1, \ldots, \sigma_n)$ is \emph{$(c_P)_{P \in \PPart(n)}$-stable} if
\begin{enumerate}
  \item $(\pi : C \to S, \sigma_1, \ldots, \sigma_n)$ is a flat and proper family
  of connected, reduced, Gorenstein curves of arithmetic genus one with $n$ distinct marked points
\end{enumerate}
and for each geometric fiber $C_s$ the following conditions hold:
\begin{enumerate}
  \item if $p \in C_s$ is an elliptic Gorenstein singularity, then $\lev(p) \in Q_{sing}$;
  \item if $Z \subseteq C_s$ is a connected subcurve of genus one, then $\lev(Z) \in Q_{curve}$;
  \item if $Z_1 \subsetneq Z_2$ is a proper inclusion of connected subcurves of $C_s$ of arithmetic genus one, then $\lev(Z_1) \prec \lev(Z_2)$;
  \item if $F$ is an irreducible component of the minimal genus one subcurve $Z_{min}$ of $C_s$, then $F$ meets $\ol{C_s - Z_{min}} \cup \{\sigma_1(s), \ldots, \sigma_n(s) \}$ in at least one point.
\end{enumerate}
 
Let $\ol{\mathcal{M}}_{1,n}((c_P))$ be the stack whose $S$-points are the $(c_P)$-stable families of curves over $S$.
\end{definition}

\begin{remark}
Note that $Q_{sing}$ is downward closed, $Q_{curve}$ is upward closed, and the two sets intersect in the partitions $P$ where $0 < c_P < 1$.

The new feature of this definition is that a $(c_P)$-stable curve $C$ may have a minimal genus one subcurve $Z$ with an elliptic
Gorenstein singularity $p$ so that $\lev(p) = \lev(Z)$, so long as $\lev(p) \in Q_{sing} \cap Q_{curve}$. If this happens, then each of the branches
of the singularity $p$ will have only one special point other than $p$. By \cite[Corollary 2.4]{smyth_mstable}, $C$ then has infinitesimal
automorphisms. (In fact, it has a $\mathbb{G}_m$ of automorphisms.) It follows that whenever $Q_{sing} \cap Q_{curve}$ is non-empty, $\ol{\mathcal{M}}_{1,n}((c_P))$ is not a Deligne-Mumford stack.
\end{remark}

\begin{theorem} \label{thm:Q_at_vertices}
If $(c_P)$ consists only of 1's and 0's, then $\ol{\mathcal{M}}_{1,n}((c_P)) = \ol{\mathcal{M}}_{1,n}(Q)$ where
\[
  Q = \{ P \in \PPart(n) \mid c_P = 1 \}.
\]
\end{theorem}

\begin{proof}
The only differently stated conditions for membership in $\ol{\mathcal{M}}_{1,n}((c_P))$ and $\ol{\mathcal{M}}_{1,n}(Q)$ are those on
geometric fibers, so it suffices to show their geometric points coincide.

Suppose $C$ is a $Q$-stable curve over some algebraically closed field. Then it is clear that conditions (i) and (ii) hold since $Q = Q_{sing}$ and $\PPart(n) - Q = Q_{curve}$. Condition (iv) holds since $C$ has no infinitesimal automorphisms.

To see that condition (iii) holds, suppose that $Z_1 \subsetneq Z_2$ is a proper inclusion of connected subcurves $C$ of arithmetic genus one.
Then there is a rational component $F$ of $Z_2$ meeting $Z_1$. Since $F$ has at least 3 special points, $\lev(Z_1) \prec \lev(Z_1 \cup F) \preceq \lev(Z_2)$.

Conversely, suppose that $C$ is $(c_P)$-stable. The level conditions for $Q$ on $C$ hold by (i) and (ii). It remains to show that $C$ has
no infinitesimal automorphisms, that is, by \cite[Corollary 2.4]{smyth_mstable}:
\begin{enumerate}
  \item each irreducible component $F$ of $C$ with genus zero not meeting an elliptic Gorenstein singularity has at least three special points;
  \item if $C$ has an elliptic Gorenstein singularity $q$, then each component $B$ of the minimal elliptic subcurve containing $q$ has at least one special point other than $q$ and there is at least such component with two special points other than $q$.
\end{enumerate}

To address (i), let $F$ be an irreducible component of $C$ with genus zero not meeting an elliptic Gorenstein singularity. Then either $C$ is nodal
and $F$ belongs to the core of $C$, or $F$ is not a component of the minimal subcurve of genus one. In the former case, we are done by (iv).
In the latter, let $Z_1$ be the union of the minimal genus one subcurve $E$ together with the components on the unique path from $E$ to $F$, not including $F$. Let $Z_2 = Z_1 \cup F$. Then, since $\lev(Z_1) \prec \lev(Z_2)$, $F$ must have at least three special points.

To address (ii), suppose that $C$ has an elliptic Gorenstein singularity $q$. Let $Z_{min}$ be the minimal genus one subcurve of $C$. If $B$ is an irreducible component of $Z_{min}$, then $B$ has at least one other special point by condition (iv).
Next, since $Q_{sing} \cap Q_{curve} = \emptyset$, we must have a proper refinement $\lev(q) \prec \lev(Z_{min})$. This implies that there is at least one branch of $q$ with two special points other than $q$.
\end{proof}

\begin{theorem}
The stack of $(c_P)$-stable curves is deformation open. That is, if $S$ is a noetherian scheme and and $\pi: C \to S$ is a flat, projective morphism
with one-dimensional fibers and sections $\sigma_1, \ldots, \sigma_n$, then the set 
\[
  T = \{ s \in S \mid (\pi_s : C_{\ol{s}} \to \ol{s}, \{ \sigma_i(\ol{s}) \}_{i = 1}^n) \text{ is $((c_P))$-stable} \}
\]
is open in $S$.
\end{theorem}
\begin{proof}
As in Theorem \ref{thm:Q_deformation_open}, we may assume that the geometric fibers $C_{\ol{s}}$ of $\pi$ are reduced, connected, and of arithmetic genus one with only Gorenstein singularities, since these are open conditions.

Again as in Theorem \ref{thm:Q_deformation_open}, the locus $T$ is constructible since satisfaction of the remaining conditions is constant on combinatorial types and the curves with a given combinatorial type form a locally closed subset of $S$. Therefore it suffices to check that the remaining conditions hold under generization.

So assume $S$ is the spectrum of a DVR with closed point $0 \in S$ and generic point $\eta \in S$. We must show that if $(C_{\ol{0}}, \sigma_1(\ol{0}), \ldots, \sigma_n(\ol{0}))$ satisfies the remaining conditions, then so does $(C_{\ol{\eta}}, \sigma_1(\ol{\eta}), \ldots, \sigma_n(\ol{\eta}))$.
Write $\Sigma$ for the divisor of markings.
Since $T$ is characterized by geometric fibers, we may apply a finite base change to $S$ so that restriction to the geometric generic fiber induces
a bijection from the components of $C$ to the components of $C_{\ol{\eta}}$.
The level conditions on singularities and subcurves are stable under generization by an identical argument to Theorem \ref{thm:Q_deformation_open}.

Next, we consider condition (iii). Suppose $Z_1^{\ol{\eta}} \subset Z_2^{\ol{\eta}}$ is a proper inclusion of genus one subcurves of $C_{\ol{\eta}}$. Let $Z_1^{\ol{0}}$ and $Z_2^{\ol{0}}$
be the respective limits of $Z_1^{\ol{\eta}}$ and $Z_2^{\ol{\eta}}$ in $C_{\ol{0}}$. Then we must have a proper inclusion $Z_1^{\ol{0}} \subset Z_2^{\ol{0}}$. Taking limits of the connected components of $(\ol{C_{\ol{\eta}} - Z_1^{\ol{\eta}}}) \cup \Sigma|_{\ol{\eta}}$, we see that $\lev(Z_1^{\ol{0}}) = \lev(Z_1^{\ol{\eta}})$. Similarly, $\lev(Z_2^{\ol{0}}) = \lev(Z_2^{\ol{\eta}})$.
By $(c_P)$-stability of the central fiber, $\lev(Z_1^{\ol{0}}) \prec \lev(Z_2^{\ol{0}})$. Since $\lev(Z_1^{\ol{0}}) = \lev(Z_1^{\ol{\eta}})$ and
$\lev(Z_2^{\ol{0}}) = \lev(Z_2^{\ol{\eta}})$, we have the desired refinement $\lev(Z_1^{\ol{\eta}}) \prec \lev(Z_2^{\ol{\eta}})$ in the generic fiber too.

Finally, we show condition (iv) is stable under generization. 
Let
\begin{align*}
  Z^{\ol{\eta}}_{min} &= \text{ minimal genus one subcurve of }C_{\ol{\eta}} \\
  Z^{\ol{0}}_{min} &= \text{ minimal genus one subcurve of }C_{\ol{0}}
\end{align*}
Given an irreducible component $F_{\ol{0}}$ of $Z^{\ol{0}}_{min}$, there is a unique irreducible component $F_{\ol{\eta}}$ of $Z_{\ol{\eta}}$ to which $F_{\ol{0}}$ generizes. Moreover, as $F_{\ol{0}}$ varies over the components of $Z^{\ol{0}}_{min}$, $F_{\ol{\eta}}$ varies over all the irreducible components of $Z^{\ol{\eta}}_{min}$.

By axiom (iv), $F_{\ol{0}}$ either meets a marking or it meets a connected rational tail $T$ contained in $\ol{C_{\ol{0}} - Z^{\ol{0}}_{min}}$. If $F_{\ol{0}}$ itself contains a marking, we set $Y^1_{\ol{0}} = F_{\ol{0}}$ and let $k = 1$. If not, then we may find a path in the dual graph of $C_{\ol{0}}$ from $F_{\ol{0}}$ to a marked component of $T$, i.e., a sequence $Y^1_{\ol{0}}, \ldots, Y^k_{\ol{0}}$ of irreducible components of $C_{\ol{0}}$, where $Y^1_{\ol{0}} = F_{\ol{0}}$, $Y^{i+1}_{\ol{0}}$ meets $Y^i_{\ol{0}}$ in a node for all $1 \leq i < k$, $Y^i_{\ol{0}} \not\subseteq Z_{\ol{0}}$ for each $i > 0$, and $Y^k_{\ol{0}}$ contains a marking.  For each $i$ we let $Y^{i}_{\ol{\eta}}$ be the unique irreducible component of $C_{\ol{\eta}}$ generizing $Y^i_{\ol{0}}$. After reindexing to omit repeats, the resulting sequence of components $Y^1_{\ol{\eta}}, \ldots, Y^l_{\ol{\eta}}$ of $C_{\ol{\eta}}$ is again a sequence of components connected by nodes, beginning with $F_{\ol{\eta}}$, and ending in a component with a marking.

If all of the nodes of the path $Y^1_{\ol{0}}, \ldots, Y^k_{\ol{0}}$ smooth out in the geometric generic fiber, then $l = 1$, and $F_{\ol{\eta}}$ meets a marking, and we are done. If not, then let $p \in C_{\ol{0}}$ be the first node on the path that does not smooth out in the family. Then, after a finite base change if necessary, there is a section $P : S \to C$ through the singular locus of $C$ going through $p$ and meeting $F_{\ol{\eta}}$. Blowing up along $P$, we see that $Y^2_{\ol{\eta}}$ does not belong to $Z^{\ol{\eta}}_{min}$, since $P$ separates the limit of $Y^2_{\ol{\eta}}$ from $Z^{\ol{0}}_{min}$ in the central fiber. Then $F_{\eta}$ meets $\ol{C_{\ol{\eta}} - F_{\eta}}$ in the point $P|_{\ol{\eta}}$, and we deduce that axiom (iv) holds
in the generic fiber.
\end{proof}

\begin{corollary}
$\ol{\mathcal{M}}_{1,n}((c_P))$ is an Artin stack.
\end{corollary}
\begin{proof}
The preceding Theorem shows that $\ol{\mathcal{M}}_{1,n}((c_P))$ is an open substack of the Artin stack of all $n$-pointed curves. (See \cite[Appendix B]{smyth_zstable}.) 
\end{proof}

Our next result is that these various moduli spaces are related by containments induced by the containment of faces in the cube complex $X_{1,n}$. This strikes us as analogous
to the containments of moduli spaces seen at critical values of the log minimal model program in \cite[Theorem 1.1]{alper_fedorchuk_smyth_flip2},
except that the ``larger" stacks here are associated with larger strata.

\begin{theorem} \label{thm:moduli_of_faces}
Suppose that $(c_P), (d_P) \in X_{1,n}$ and $(d_P)$ lies in a face of the cube containing $(c_P)$, i.e., $d_P = c_P$ whenever $c_P = 0$ or $c_P = 1$. Then
there is a fully faithful inclusion functor
\[
  \ol{\mathcal{M}}_{1,n}((d_P)) \hookrightarrow \ol{\mathcal{M}}_{1,n}((c_P)).
\]
\end{theorem}
\begin{proof}
Note that the sets $Q_{sing}$ and $Q_{curve}$ associated to $(d_P)$ are subsets of the respective sets associated to $(c_P)$. Then the result is clear from the definition.
\end{proof}

\begin{corollary}
The stacks $\ol{\mathcal{M}}_{1,n}((c_P))$ are universally closed.
\end{corollary}
\begin{proof}
Let $Q = \{ P \in \Part(n) \mid c_P = 1 \}$. By Theorem \ref{thm:moduli_of_faces} and Theorem \ref{thm:Q_at_vertices}, $\ol{\mathcal{M}}_{1,n}(Q)$ is a substack of $\ol{\mathcal{M}}_{1,n}((c_P))$.

Note that $\ol{\mathcal{M}}_{1,n}((c_P))$ contains $\mathcal{M}_{1,n}$ as a dense open, since nodes and elliptic Gorenstein singularities are smoothable. To check universal closedness it therefore suffices to
verify that if $S$ is the spectrum of a DVR with generic point $\eta$, and $C_\eta \to \eta$ a family of smooth $n$-marked curves of genus one, then there is an extension of $C_\eta$ to a family of $(c_P)$-stable
curves over $S$. Such an extension already exists as a $Q$-stable limit.
\end{proof}

\subsection{A tropical resolution of the birational map from $\ol{\mathcal{M}}_{1,n}$ to $\ol{\mathcal{M}}_{1,n}((c_P))$}

We now give the tropical construction that motivated the definition of $\ol{\mathcal{M}}_{1,n}((c_P))$ above.
For each $(c_P) \in X_{1,n}$ we will build a modification $\widetilde{\mathcal{M}}_{1,n}^{rad} \to \ol{\mathcal{M}}_{1,n}^{rad}$ and a contraction inducing
a morphism $\widetilde{\mathcal{M}}_{1,n}^{rad} \to \ol{\mathcal{M}}_{1,n}((c_P))$. The ideas are essentially the same as in Section \ref{sec:radially_aligned_contractions}, so we are relatively brief.

Fix an element $(c_P)_{P \in \PPart(n)} \in X_{1,n}$. Let $Q_{min} = \{ P \in \PPart(n) \mid c_P = 1 \}$. Let $P_1, \ldots, P_k$ be the partitions for
which $0 < c_{P_i} < 1$. Then let $Q_i = Q_{min} \cup \{ P_i \}$ for $1 \leq i \leq k$.
Consider the associated global sections $\rho_{min}, \rho_1, \ldots, \rho_k$ of the characteristic sheaf of $\ol{\mathcal{M}}_{1,n}^{rad}$. Notice that $\rho_i \geq \rho_{min}$ for each $i$, so the differences $\delta_i = \rho_{i} - \rho_{min}$ are again well-defined sections of the characteristic sheaf
of $\ol{\mathcal{M}}_{1,n}^{rad}$.

We recall (see for example \cite[Proposition 5.17]{olsson} with $P = \mathbb{N}$) that the stack $[\mathbb{A}^1 / \mathbb{G}_m]$ may be given a log structure so that $[\mathbb{A}^1 / \mathbb{G}_m]$ represents the functor on log
schemes
\[
  X \mapsto \Gamma(X, \ol{M}_X).
\]
Since this is a functor valued in commutative monoids, there is a morphism $\mu : [\mathbb{A}^1/\mathbb{G}_m] \times [\mathbb{A}^1/\mathbb{G}_m] \to [\mathbb{A}^1/\mathbb{G}_m]$ induced by the multiplication of $\Gamma(X, \ol{M}_X)$.

Take products and form the pullback square of fs log algebraic stacks
\[
\xymatrixcolsep{5pc}
\xymatrix{
  \widetilde{\mathcal{M}}_{1,n}^{rad} \ar[r]^-{\prod_i (\delta_i^{(1)}, \delta_i^{(2)})} \ar[d]_{p} & \left([\mathbb{A}^1/\mathbb{G}_m] \times [\mathbb{A}^1/\mathbb{G}_m]\right)^k \ar[d]^{\mu^k} \\
  \ol{\mathcal{M}}_{1,n}^{rad} \ar[r]^-{\prod_i \delta_i} & [\mathbb{A}^1/\mathbb{G}_m]^k.
}
\]
Note that $p^*\delta_i$ factors into the sum $\delta_i^{(1)} + \delta_i^{(2)}$ in the characteristic sheaf for each $i$. Moreover, $\widetilde{\mathcal{M}}_{1,n}^{rad}$ is universal with respect to such factorizations in the sense that, given any $g : T \to \ol{\mathcal{M}}_{1,n}^{rad}$
and an expression of each $g^*\delta_i$ as a sum $\alpha_i^{(1)} + \alpha_i^{(2)}$, there is a unique factorization $h : T \to \widetilde{\mathcal{M}}_{1,n}^{rad}$ of $g$ through $\widetilde{\mathcal{M}}_{1,n}^{rad}$ so that $\alpha_i^{(1)} = h^*\delta_i^{(1)}$ and $\alpha_i^{(2)} = h^*\delta_i^{(2)}$ for each $i$.

The morphism $\mu$ is integral and saturated, so the underlying algebraic stack of $\widetilde{\mathcal{M}}_{1,n}^{rad}$
is the fiber product of the underlying algebraic stacks of the rest of the diagram. The complement $D(\delta_1,\ldots,\delta_k)$ in $\ol{\mathcal{M}}^{rad}_{1,n}$ of the vanishing of the Cartier divisors associated to
$\delta_1, \ldots, \delta_k$ is precisely the preimage of the non-stacky point of $[\mathbb{A}^1/\mathbb{G}_m]^k$. As $\mu^k$ restricts to an isomorphism over this point, it follows that $p$ restricts to an isomorphism over $D(\delta_1,\ldots,\delta_k)$. In particular, since the $\delta_i$'s are non-vanishing on smooth curves, $p$ restricts to an isomorphism on $\mathcal{M}_{1,n}$. We identify $\mathcal{M}_{1,n}$ with its image in $\widetilde{\mathcal{M}}_{1,n}^{rad}$.

\begin{lemma} \label{lem:weird_ring_irreducible}
Let $T = \Spec B$ be the spectrum of a discrete valuation ring with uniformizer $t$ and field of fractions $K$.
Let $b_1,\ldots, b_k$ be nonzero elements of $B$, and let
\[
  A = \frac{B[x_1^{(1)},x_1^{(2)}, \ldots, x_k^{(1)}, x_k^{(2)}]}{(x_i^{(1)}x_i^{(2)} - b_i : i = 1,\ldots, k)}.
\]
Then $S = \Spec A$ is an integral scheme and its generic point maps to the generic point of $T$ under the natural map $S \to T$.
\end{lemma}

\begin{proof}
Let $S = \Spec A$, $T = \Spec B$. Let $\eta = D(t)$ be the generic point of $T$ and $U$ its preimage in $S$. Then $U \cong \G_{m,K}^k$ is an irreducible open of $S$. Recall that the multiplication map $\A^1 \times \A^1 \to \A^1$ is flat, so its $k$-fold product
$(\A^1 \times \A^1)^k \to (\A^1)^k$ is also flat. Observe that $S \to T$ is a pullback of this morphism, so also flat. If $x$ is any point of $S$ not in $U$, the Going Down Theorem for flat morphisms implies that $x$ possesses a generization in $U$. Therefore, $U$ is dense in $S$, and $S$ is integral. Since $U$ maps to the generic point of $T$, we also have the claim about generic points.
\end{proof}

\begin{proposition}
The open substack $\mathcal{M}_{1,n}$ is dense in $\widetilde{\mathcal{M}}_{1,n}^{rad}$. In particular $\widetilde{\mathcal{M}}_{1,n}^{rad}$ is irreducible and $p$ is birational.
\end{proposition}

\begin{proof}
Give $\A^1$ the toric log structure. Recall that $\A^1$ represents the functor
\[
  (X, \alpha : M_X \to \mathscr{O}_X) \mapsto \Gamma(X, M_X).
\]
There is a commutative square of fs log algebraic stacks
\[
\xymatrix{
  ([\A^1/\G_m] \times [\A^1/\G_m])^k \ar[d]^{\mu^k} & \ar[l] (\A^1 \times \A^1)^k \ar[d] \\
  [\A^1/\G_m]^k & (\A^1)^k \ar[l]
}
\]
where each vertical map is induced by the monoid law, and the horizontal maps are induced by $M_X \to \ol{M}_X$. The map of schemes underlying the right vertical map consists of $k$-copies of the multiplication map. By \cite[Proposition 2.1]{olsson},
a map $X \to [\A^1/\G_m]$ \'etale locally admits a lift to $\A^1$.

Let $s = (\Spec k, \alpha)$ be a geometric fs log point, and let $f_s : s \to \widetilde{\mathcal{M}}_{1,n}^{rad}$ be a log morphism. Our strategy is to find a smoothing of $f_s$ in $\ol{\mathcal{M}}_{1,n}$, then to lift this smoothing back to $\widetilde{\mathcal{M}}_{1,n}^{rad}$ using the commutative square above.

Observe that since $s$ is a geometric point,
the composite $s \to \widetilde{\mathcal{M}}_{1,n}^{rad} \to ([\A^1/\G_m]\times[\A^1/\G_m])^k$ factors through $(\A^1 \times \A^1)^k$.

Let $t = (\Spec k, \beta)$ denote the fs log point with log structure pulled-back from $\ol{\mathcal{M}}_{1,n}^{rad}$ along $p \circ f_s : s \to \ol{\mathcal{M}}_{1,n}^{rad}$.
Write $g_t$ for the induced map $t \to \ol{\mathcal{M}}_{1,n}^{rad}$. Note that we have a commutative square of fs log algebraic stacks
\[
  \xymatrix{
    s \ar[r]^-{f_s} \ar[d] & \widetilde{\mathcal{M}}_{1,n}^{rad} \ar[d]^p \\
    t \ar[r]^-{g_t} & \ol{\mathcal{M}}_{1,n}^{rad}.
  }
\]
By construction, $g_t$ is strict. We may therefore use that $\mathcal{M}_{1,n}$ is dense in the Noetherian algebraic stack $\ol{\mathcal{M}}_{1,n}^{rad}$ to find a strict map $g : T \to \ol{\mathcal{M}}_{1,n}^{rad}$ where
\begin{enumerate}
  \item $T$ is an fs log scheme with underlying scheme the spectrum of a discrete valuation ring;
  \item $g_t$ factors through the special point of $T$;
  \item $\eta \subseteq T$ is the generic point;
  \item $g|_\eta \in \mathcal{M}_{1,n}(\eta)$.
\end{enumerate}
Replacing $T$ by a finite base change if necessary, we may assume that the
composite $T \to \ol{\mathcal{M}}_{1,n}^{rad} \to [\mathbb{A}^1/\mathbb{G}_m]^k$ factors through $(\mathbb{A}^1)^k$.
A diagram chase shows that there is an induced map $f : S \to \widetilde{\mathcal{M}}_{1,n}^{rad}$ and $f_s : s \to \widetilde{\mathcal{M}}_{1,n}^{rad}$ factors through it.
By Lemma \ref{lem:weird_ring_irreducible}, $S$ is irreducible and its
generic point $\theta$ maps to the generic point $\eta$ of $T$. By construction, $g(\eta)$ lies in the smooth locus of $\ol{\mathcal{M}}_{1,n}^{rad}$, so $f(\theta)$ lies in the smooth locus of $\widetilde{\mathcal{M}}_{1,n}^{rad}$. It follows that the image of $s$ in $\widetilde{\mathcal{M}}_{1,n}^{rad}$ has a generalization to a point factoring through $\mathcal{M}_{1,n} \subseteq \widetilde{\mathcal{M}}_{1,n}^{rad}$, as desired.
\end{proof}

Now, let $\rho = p^*\rho_{min} + \delta_1^{(1)} + \cdots + \delta_k^{(1)}$ in the characteristic sheaf of $\widetilde{\mathcal{M}}_{1,n}^{rad}$.
Let $C_{1,n} \to \widetilde{\mathcal{M}}_{1,n}^{rad}$ be the pullback of the universal family of $\ol{\mathcal{M}}_{1,n}^{rad}$ to $\widetilde{\mathcal{M}}_{1,n}$
Notice that $\rho$ is comparable with the radii of $C_{1,n}$, so we may make a log modification $\widetilde{C}_{1,n} \to C_{1,n}$ subdividing tropicalizations at
the locus where $\rho = \lambda$. Then, as in Theorem \ref{thm:contraction}, we may form a section $\ol{\lambda} \in \Gamma(\widetilde{C}, \ol{\mathcal{M}}_{\widetilde{C}})$ by the formula
\[
  \ol{\lambda} = \max\{ \rho - \lambda, 0 \}.
\]
Once again, it is easy to check that this is a mesa in the sense of \cite{bozlee_thesis}, so the main result of \cite{bozlee_thesis} yields a contraction
of families of curves
\[
  \xymatrix{
    \widetilde{C}_{1,n} \ar[r]^\tau \ar[d]_{\pi} & \ol{C}_{1,n} \ar[dl]^{\ol{\pi}} \\
    \widetilde{\mathcal{M}}_{1,n}^{rad}
  }
\]
Similar reasoning to that of Section \ref{sec:radially_aligned_contractions} yields that $\ol{C}_{1,n} \to \widetilde{\mathcal{M}}_{1,n}^{rad}$
is a family of curves in $\ol{\mathcal{M}}_{1,n}((c_P))$. We therefore have a diagram of birational morphisms of algebraic stacks
\[
\xymatrix{
 & & \widetilde{\mathcal{M}}_{1,n}^{rad} \ar[dr]^{\ol{C}_{1,n}} \ar[dl]_{p} & \\
 & \ol{\mathcal{M}}_{1,n}^{rad} \ar[dl] & & \ol{\mathcal{M}}_{1,n}((c_P))\\
  \ol{\mathcal{M}}_{1,n} & & & \\
}.
\]

\bibliographystyle{amsalpha}
\bibliography{references}

\end{document}